\documentclass[final,3p,times]{elsarticle}

\usepackage{amssymb}
\usepackage{amsthm}

\usepackage{latexsym}
\usepackage{amsmath}
\usepackage{color}
\usepackage{graphicx}
\usepackage{indentfirst}
\usepackage{mathrsfs}
\usepackage{pdfsync}
\usepackage{hyperref}

\hypersetup{hypertex=true,
            colorlinks=true,
            linkcolor=blue,
            anchorcolor=blue,
            citecolor=blue}
\biboptions{sort&compress}
\allowdisplaybreaks
\usepackage[utf8]{inputenc}

\newtheorem{theorem}{\color{black}\indent Theorem}
\newtheorem{lemma}{\color{black}\indent Lemma}[section]
\newtheorem{proposition}{\color{black}\indent Proposition}

\newtheorem{remark}{\color{black}\indent Remark}[section]

\DeclareMathOperator{\diag}{{diag}}

\journal{a}

\begin{document}

\begin{frontmatter}



\title{Melnikov's persistence for completely degenerate Hamiltonian systems}

\author{{Jiayin Du$^{a}$ \footnote{ E-mail address : dujy668@nenu.edu.cn} ,~Shuguan Ji$^{a,*}$ \footnote{E-mail address : jisg100@nenu.edu.cn},~Yong Li$^{b,a}$ \footnote{E-mail address : liyong@jlu.edu.cn\\\indent ~*Corresponding author} }\\
{$^{a}$School of Mathematics and Statistics, and Center for Mathematics and Interdisciplinary Sciences, Northeast Normal
University, Changchun 130024, P.R.  China.}\\
{$^{b}$School of Mathematics, Jilin University, Changchun 130012, P.R. China.}
 }

\begin{abstract}
In this paper, we study the Melnikov's persistence for completely degenerate Hamiltonian systems with the following Hamiltonian
\begin{equation*}
H(x,y,u,v)=h(y)+g(u,v)+\varepsilon P(x,y,u,v),~~~(x,y,u,v)\in \mathbb{T}^n\times{G}\times \mathbb{R}^d\times \mathbb{R}^d,
\end{equation*}
where $n\geq2$ and $d\geq1$ are positive integers, $G\subset\mathbb{R}^n$, $g=o(|u|^2+|v|^2)$ admits complete degeneracy and certain transversality, and $\varepsilon P$ is the small perturbation.
This is a try in studying lower-dimensional invariant tori in the normal complete degeneracy.
Under {R\"{u}ssmann-like} non-degenerate condition and transversality condition, we apply the homotopy invariance of topological degree to remove the first order terms about $u$ and $v$ and employ the quasi-linear KAM iterative procedure to derive the persistence of lower-dimensional invariant tori.
\end{abstract}

\begin{keyword}
Hamiltonian system, complete degeneracy, Melnikov's persistence.
\end{keyword}





\end{frontmatter}



\tableofcontents

\section{Introduction}\label{sec:1}
In this paper, we consider the following Hamiltonian system \begin{equation}\label{eq1}
H(x,y,u,v)=h(y)+g(u,v)+\varepsilon P(x,y,u,v),
\end{equation}
where $(x,y,u,v)\in \mathbb T^n\times{G}\times \mathbb{R}^d\times \mathbb{R}^d$, $n\geq2$ and $d\geq1$ are positive integers, $\mathbb T^{n}$ is the standard $n$-torus, $G\subset{\mathbb{R}^n}$ is a bounded closed region,  $g=o\left(|u|^2+|v|^2\right)$ and is therefore completely degenerate; $h$, $g$ and $P$ are real analytic functions in $(x,y,u,v)$, $\varepsilon>0$ is a small parameter and $\varepsilon P$ is the small perturbation. We will explore the persistence of lower-dimensional invariant tori for the completely degenerate Hamiltonian system under certain high degeneracy conditions. {Such Melnikov's persistence in the degenerate cases is very difficult to study.}

Let's review some developments on Melnikov's persistence. Consider the normal form:
$$N=\langle\omega,y\rangle+\frac{1}{2}\langle z,\Omega  z\rangle,$$
where $(y,z)\in \mathbb{R}^n\times \mathbb{R}^{2d}$, $n\geq2$, $d\geq1$. In non-degenerate case, i.e., $\Omega$ is non-singular, if all the eigenvalues of $J\Omega$ belong to $i\mathbb{R}^1\backslash\{0\}$ ($J$ is the standard symplectic matrix), we say that the lower-dimensional invariant torus is elliptic, which corresponds to the following normal form:
\begin{align*}
N=\langle\omega,y\rangle+\frac{1}{2}\sum_{j\geq1}^d\Omega_j\left(u_j^2+v_j^2\right),
\end{align*}
where $(y,z)\in \mathbb{R}^n\times \mathbb{{R}}^{2d}$, $z=(u,v)\in\mathbb{{R}}^{d}\times\mathbb{{R}}^{d}$.
In 1965, Melnikov \cite{melnikov} announced that the elliptic lower-dimensional invariant tori can persist for any small perturbations under certain coupling nonresonance conditions, called Melnikov conditions:
\begin{align*}
\langle k,\omega\rangle+\langle l,\Lambda\rangle\neq0,~~~~\forall (k,l)\in\left(\mathbb{Z}^n\times\mathbb{Z}^{2d}\right)\setminus\{0\},~~|l|\leq2,
\end{align*}
where $\Lambda$ is the column vector formed by the eigenvalues of $\Omega$. A detailed and rigorous proof of the above Melnikov's theorem was due to Eliasson \cite{Eliasson} and then an infinite dimensional analogue was obtained by Kuksin \cite{kuksin1}, P\"{o}schel \cite{poschel}, Bourgain \cite{bourgain1} and Wayne \cite{wayne}. On the other hand, if all the eigenvalues of $J\Omega$ are not on the imaginary axis, the lower-dimensional invariant torus is called hyperbolic, which is associated with the unperturbed system
\begin{align*}
N=\langle\omega,y\rangle+\frac{1}{2}\sum_{j\geq1}^d\Omega_j\left(u_j^2-v_j^2\right).
\end{align*}
The persistence problem for hyperbolic tori was first considered in the work of Moser \cite{moser1} and later studied by Graff \cite{graff} and Zehnder \cite{zehnder75,zehnder76} for a fixed toral frequency $\omega$ satisfying the following Diophantine condition: \begin{align}\label{dio}
|\langle k,\omega\rangle|>\frac{\gamma}{|k|^\tau},~~~k\in{\mathbb{Z}^n\setminus{\{0\}}},
\end{align}
where $k=(k_1,\cdots,k_n)$, $|k|=|k_1|+\cdots+|k_n|$, $\gamma>0$ and $\tau> n-1$.
Moreover, the persistence of lower-dimensional invariant tori has been extensively studied in various (normally) non-degenerate cases, see \cite{bourgain,cheng1,chier,Eliasson,gallatti,melnikov,poschel,qian3,xulu}  and references therein.

However, in the degenerate cases, i.e., $\Omega$ is singular, generally, {the lower-dimensional invariant tori might be destroyed under small perturbations.} Then a natural question is what degenerate systems admit Melnikov's persistence, which is an essential problem, for example, see \cite{kuksin}. A normal form in the simplest degenerate case is, see Takens \cite{takens},
$$N=\langle\omega,y\rangle+\frac{1}{2}v^2+f(u,v),$$ where $y\in \mathbb{R}^n$, $f(u,v)=o(|u|^2+|v|^2)\neq0$, $u,v\in \mathbb{R}^1$. You in \cite{you} proved the persistence of a hyperbolic-type degenerate lower-dimensional torus for Hamiltonian systems with the following normal form
$$N=\langle\omega,y\rangle+\frac{1}{2}v^2-u^{2l},~~~l\geq2,$$
where $(y,u,v)\in \mathbb{R}^n\times \mathbb{R}^1\times \mathbb{R}^1$. For more details see \cite{xu}. Li and Yi \cite{li} showed the persistence of lower-dimensional tori of general types for the Hamiltonian system with the following general normal form
$$N=e+\langle\omega,y\rangle+\frac{1}{2}\left\langle\left(\begin{array}{c}
y\\
z
\end{array}\right),M\left(\begin{array}{c}
y\\
z
\end{array}\right)\right\rangle+O\left(|(y,z)|^3\right),$$
where $(y,z)\in \mathbb{R}^n\times \mathbb{R}^{2d}$, $d\geq1$, and $M$ is nonsingular. Furthermore, Han, Li and Yi \cite{han} study degenerate Hamiltonian system with the following normal form
\begin{align*}
N=e+\langle\omega,y\rangle+\frac{1}{2}\langle z,M(\omega)z\rangle+\varepsilon P(x,y,z,\omega),
\end{align*}
where $(y,z)\in\mathbb{R}^n\times \mathbb{R}^{2d}, d\geq1$, and $M(\omega)$ can be singular. They impose some conditions on perturbation $P$ to remove the singularity of $M(\omega)$ and hence yield the persistence of the majority of lower-dimensional invariant tori. Recently, Hu and Liu \cite{hu} investigate completely degenerate Hamiltonian system with the following norm form
$$N=\langle\omega,y\rangle+\frac{1}{2p}u^{2p}+\frac{1}{2q}v^{2q}+y_1^mu+y_2^lv,$$
where $(y,u,v)\in\mathbb{R}^n\times\mathbb{R}^1\times\mathbb{R}^1$, $p, q>1$ and $m, l$ are positive integers. As they pointed out that the high order terms $y_1^mu$ and $y_2^lv$ are needed to control the normal directions.


Obviously, in the cases we are considering system (\ref{eq1}) is completely degenerate and $g=o\left(|u|^2+|v|^2\right)$. We will study the Melnikov's persistence without any restriction on
perturbation but only smallness and analyticity. So a fundamental question is what conditions do we need to add to the system to ensure persistence? To this aim, we have done the following:

\emph{We construct some sufficient conditions in term of the topological degree condition as well as the weak convexity condition for $g(u,v)$, see {(A0)} in Section \ref{sec:2} for details of these two conditions. We use the technique of translating normal direction to remove the first order term about $u$ and $v$. We employ the quasi-linear KAM iterative procedure to prove the persistence of lower-dimensional invariant tori. We also give an example to state that our condition {(A0)} is {sharp} to Melnikov's persistence, see Proposition \ref{pro2}.}

The paper is organized as follows. In Section \ref{sec:2}, we state our main results (Theorem \ref{th1}, Proposition \ref{pro1} and Proposition \ref{pro2}). Section \ref{sec:3} contains the detailed construction and estimates for one cycle of KAM steps. In Section \ref{sec:4}, we complete the proof of Theorem \ref{th1} by deriving an iteration lemma and showing the convergence of KAM iterations. Finally, the proof of Proposition \ref{pro1} and Proposition \ref{pro2} can be found in Appendix A and Appendix B, respectively.

\section{Main results}\label{sec:2}
\setcounter{equation}{0}
To state our main results, we first need to introduce a few definitions and notations.
\begin{itemize}
\item[{(1)}] Given a domain $D$. Let $\bar{D}$, $\partial D$ denote the closure of $D$ and the boundary of $D$,  respectively. $D^o:=\bar{D}\setminus\partial D$ refers to the interior.
\item[{(2)}] {We shall use the same symbol $|\cdot|$ to denote the Euclidean norm of vectors, and use $|\cdot|_D$ to denote the supremum norm of a function on a domain $D$.}
\item[{(3)}] For any two complex column vectors $\xi$, $\eta$ of the same dimension, $\langle\xi,\eta\rangle$ always stands for $\xi^\top\eta$, i.e., the transpose of $\xi$ times $\eta$.
\item[{(4)}] {$id$ is the identical mapping, and $I_d$ is the $d$-order unit matrix.}
\item[{(5)}] All Hamiltonians in the sequel are endowed with the standard symplectic structure.
\item[{(6)}] {$B_\delta(\zeta)$ stands for a sphere with $\zeta$ as center and $\delta$ as radius.}
\item[{(7)}] Denote the complex neighborhood of $\mathbb{T}^n\times\{0\}\times\{0\}$
    \begin{align}\label{D}
D(s,r)=\left\{(x,y,z):|\textrm{Im}x|<r,|y|<s^2,|z|<s\right\}
\end{align}
with $0<s,r<1$.

\end{itemize}

Let $\xi\in G$ be arbitrarily given. The Taylor expansion of Hamiltonian (\ref{eq1}) about $\xi$ {in a small neighborhood of $\xi$ in $G$} reads
\begin{align*}
H(x,y,z,\xi)=e(\xi)+\langle\omega(\xi),y-\xi\rangle+\tilde{h}(y-\xi)+g(z)+\varepsilon P(x,y,z),
\end{align*}
where $z=(u,v)\in \mathbb{R}^{2d}$,  $e(\xi)=h(\xi)$, $\omega(\xi)=\frac{\partial h(\xi)}{\partial y}$, and $\tilde{h}(y-\xi)=O(|y-\xi|^2)$. Using the
transformation $(y-\xi)\rightarrow y$ in the above, we have
\begin{align*}
H(x,y,z,\xi)=e(\xi)+\langle\omega(\xi),y\rangle+\tilde{h}(y,\xi)+g(z)+\varepsilon P(x,y,z,\xi),
\end{align*}
where $(x,y,z)\in D(s,r)$, $\xi\in G$.

We are finally led to consider the following Hamiltonian
\begin{equation}\label{Hxi}
\left\{
\begin{array}{ll}
H:\mathbb{T}^n\times G\times \mathbb{R}^{2d}\times G\rightarrow \mathbb{R}^1,\\
H(x,y,z,\xi)=e(\xi)+\langle\omega(\xi),y\rangle+\tilde{h}(y,\xi)+g(z)+\varepsilon P(x,y,z,\xi).
\end{array}
\right.
\end{equation}
First, we make the following assumptions:
\begin{itemize}
\item[(A0)] For fixed $\zeta_0=0\in \mathcal{O}^o\subset \mathbb{R}^{2d}$, where $\mathcal{O}$ is a bounded closed domain, there are $\sigma>0$, $L\geq2$ such that
\begin{align*}
&\deg\left(\nabla{g(z)}-\nabla{g(\zeta_0)}, \mathcal{O}^o, 0\right)\neq0,\\
&|\nabla{g(z)}-\nabla{g(z_*)}|\geq\sigma|z-z_*|^L,~z,z_*\in \mathcal{O},
\end{align*}
where $\nabla g(z)=(\partial_{z_1}g(z),\partial_{z_2}g(z),\cdots,\partial_{z_{2d}}g(z))$.
\item[(A1)] There is a positive integer $M$ such that
$$rank\left\{\partial_\xi^iA_0(\xi):0\leq|i|\leq M\right\}=1,  ~~~for ~all~ \xi\in G,$$\\
where $$A_0(\xi)=\left\langle\frac{k}{|k|},\omega(\xi)\right\rangle,~k\in\mathbb Z^n\backslash\{0\},~|k|=|k_1|+\cdots+|k_n|.$$
\end{itemize}

We now state the main results:
\begin{theorem}\label{th1}
Consider Hamiltonian (\ref{Hxi}).
Assume that $({A0})$-$({A1})$ hold and $\zeta_0=0$ is a equilibrium point of $g(z)$, i.e.,
\begin{align}\label{eq62}
g(z)=o(|z|^2),~~~\nabla g(\zeta_0)=0.
\end{align}
Let $\varepsilon_0>0$ be sufficiently small.
Then as $0<\varepsilon<\varepsilon_0$, there exist Cantor sets {$G_*\subset G$} with {$|G\backslash G_*|\rightarrow0$} as $\varepsilon\rightarrow0$ as well as a ${C}^{m-1}$ Whitney smooth family of symplectic transformations
$$\Psi^*=\Phi_1\circ\Phi_2\circ\cdots\circ\Phi_\infty:D\left(\frac{s}{2},\frac{r}{2}\right)\rightarrow D(s,r),~~~\xi\in G_*,$$
which is analytic in $x$, ${C}^{m-1}$ Whitney smooth in $y$, $z$ and $\xi$, and close to the identity, such that
\begin{align}\label{H*}
H_*=H\circ\Psi^*=e_*(\xi)+\langle\omega_*(\xi),y\rangle+\tilde{h}_*(y,\xi)+g_*(z,\xi)+{\bar{g}_*(y, z,\xi)}+P_*(x,y,z,\xi)
\end{align}
with
\begin{align*}
&\tilde h_*(y,\xi)=O(|y|^2),\\
&g_*(z,\xi)=g(z)+\varepsilon^{\frac{(m-2)(5m+4)}{8m(m+1)}} O(|z|^2),\\
&\bar g_*(y,z,\xi)=\sum_{2|\imath|+|\jmath|\leq m,1\leq |\imath|, |\jmath| }\bar g_{\imath\jmath}(\xi)y^\imath z^\jmath,\\
&\partial_y^i\partial_z^jP_*\big|_{(y,z)=(0,0)}=0,
\end{align*}
where $e_*, \omega_*, \tilde h_*, g_*, \bar g_*, P_*$ are real analytic in their associated variables, $2|i|+|j|\leq m$, and \begin{align}\label{m}
{{m\geq \frac{L+\sqrt{L^2+16L+16}}{4},~~ L~\mbox{is defind as in }(A0).}}
\end{align}
Thus for each $0<\varepsilon<\varepsilon_0$, Hamiltonian system (\ref{Hxi}) admits a Whitney smooth family of real analytic, quasi-periodic $n$-tori $\mathbb{T}_\xi^\varepsilon$, $\xi\in G_*$.
\end{theorem}
\begin{remark} Recall $z=(u,v)\in\mathbb R^d\times\mathbb R^d$. For partially degenerate Hamiltonian (\ref{eq1}) with
$$g(u,v)=\frac{1}{2}\sum_{i=1}^{d}\lambda_i v_i^2+f(u),~~\lambda_i\neq0,i=1,\cdots,d,d\geq1,~f(u)=O(|u|^4),$$
we will study it in a forthcoming paper.
\end{remark}
\begin{remark}
(\ref{H*}) in Theorem \ref{th1} indicates that Hamiltonian system (\ref{Hxi}) is conjugated to a nonlinear system, not a linear one. Moreover, in order to ensure the persistence of low-dimensional invariant tori, the conjugated system cannot have first order terms about $z$, i.e., $g_*(z)=O(|z|^2)$, otherwise there is no invariant torus because one cannot get  invariant set in the $z$ direction as $\nabla g_*(0)\neq0$.
\end{remark}
Indeed, as an application of Theorem \ref{th1}, we have the following proposition:
\begin{proposition}\label{pro1}
Consider Hamiltonian (\ref{Hxi}) with
\begin{align*}
g(u,v)=\frac{1}{2l_0}|u|^{2l_0}+\frac{1}{2k_0}|v|^{2k_0},
\end{align*}
where $l_0,k_0>1$ are integers, the dimension of $u$ and $v$ might be any positive integer. Assume that $\omega(\xi)$ satisfies condition {(A1)}.
Then the perturbed Hamiltonian system (\ref{eq1}) admits a family of lower-dimensional invariant tori.
\end{proposition}

See Appendix A for the complete proof.

Next, we will give an example to state that our assumption {(A0)} is sharp to Melnikov's persistence. See below for  a counter example:
\begin{proposition}\label{pro2}
Consider Hamiltonian (\ref{eq1}) with
\begin{align}\label{HH}
H=\omega\cdot y+\frac{1}{3}u^3+\frac{1}{3}v^3+\varepsilon^2 u,
\end{align}
where $y\in\mathbb{R}^1$, $(u,v)\in\mathbb{R}^1\times\mathbb{R}^1$, and $\varepsilon^2 u$ is a perturbation term. Assume $\omega$ satisfies Diophantine condition (\ref{dio}). Then the Hamiltonian system does not admit any lower-dimensional torus.
\end{proposition}

The proof can be found in Appendix B.

\section{KAM step}\label{sec:3}
In this section, we will show the detailed construction and estimates for one cycle of KAM steps, see \cite{dunon,duprm,han,li2,li,poschel}.

\subsection{Description of the 0-th KAM step.}\label{subsec:3}


Recall the integer $m$ satisfying
\begin{align}\label{m}
{{m\geq \frac{L+\sqrt{L^2+16L+16}}{4},}}
\end{align}
where $L\geq2$ is defined as in {(A0)}.
Denote $\rho=\frac{1}{2(m+1)}$, and let $\eta>0$  be an integer such that $(1+\rho)^\eta>2$. We put
\begin{equation}\label{gamma}\gamma=\varepsilon^{\frac{1}{2m(m+1)(2d)^m}}.\end{equation}
Consider the perturbed Hamiltonian $(\ref{Hxi})$. First we define the following $0$-th KAM step parameters:
\begin{align}
&r_0=r,~~~\gamma_0=\gamma,~~~\mu_0=\varepsilon^{\frac{1}{8(m+1)}},~~~s_0=s\varepsilon^{\frac{1}{8(m+1)}}\gamma_0^{(m+1)(2d)^m},~~\sigma_0=\sigma,~~\beta_0=s,\notag\\\label{h0y}
&e_0=e(\xi),~~~~\omega_0=\omega(\xi),~~~~\tilde{h}_0(y)=\tilde{h}(y),~~~~g_0(z)=g(z),~~~~\bar{g}_0=0,~~~G_0={G},\\
&D(s_0,r_0):=\left\{(x,y,z):|\textrm{Im}x|<r_0,|y|<s_0^2,|z|<s_0\right\},\notag
\end{align}
where $0<r_0,s_0,s,\gamma_0\leq 1$, $\sigma$ is defined in ${(A0)}$.
Therefore, we have that
\begin{align*}
H_0&=: H(x,y,z,\xi)=N_0+P_0,\\
N_0&=: e_0+\langle\omega_0,y\rangle+\tilde{h}_0(y)+g_0(z)+\bar{g}_0,\\
P_0&=: \varepsilon P(x,y,z,\xi).
\end{align*}
We first prove an estimate.
\begin{lemma}
Assume that \begin{equation*}
\textsc{(H0)}: \varepsilon_0^{\frac{m}{8(m+1)}}\frac{\left|\partial_\xi^\ell P\right|_{D(s_0,r_0)}}{s^m}\leq 1,~~|\ell|\leq m.
\end{equation*}
For $|\ell|\leq m$,
\begin{equation}\label{P0}
\left|\partial_\xi^\ell P_0\right|_{D(s_0,r_0)}\leq\gamma_0^{(m+1)(2d)^m}s_0^m\mu_0.
\end{equation}
\end{lemma}
\begin{proof}
Using the fact $\gamma_0=\varepsilon^{\frac{1}{2m(m+1)(2d)^m}}$, $s_0=s\varepsilon^{\frac{1}{8(m+1)}}\gamma_0^{(m+1)(2d)^m}$ and $\mu_0=\varepsilon^{\frac{1}{8(m+1)}}$, we have
\begin{align*}
s_0^m=s^m\varepsilon^{\frac{m}{8(m+1)}}\gamma_0^{m(m+1)(2d)^m}=s^m\varepsilon^{\frac{m}{8(m+1)}+\frac{1}{2}}, 
\end{align*}
and
\begin{align}\label{u0}
\gamma_0^{(m+1)(2d)^m}s_0^m\mu_0&=\varepsilon^{\frac{1}{2m}}s^m\varepsilon^{\frac{m}{8(m+1)}+\frac{1}{2}}\varepsilon^{\frac{1}{8(m+1)}}
\geq s^m\varepsilon^{\frac{1}{4}+\frac{5}{8}+\frac{1}{8(m+1)}}.
\end{align}
Then, by {(H0)}, for any $0<\varepsilon<\varepsilon_0$, we get $$\varepsilon^{\frac{m}{8(m+1)}}\frac{\left|\partial_\xi^\ell P\right|_{D(s_0,r_0)}}{s^m}\leq 1,~~|\ell|\leq m, $$
i.e.,
\begin{align}\label{P}
\varepsilon^{\frac{m}{8(m+1)}}\left|\partial_\xi^\ell P\right|_{D(s_0,r_0)}\leq s^m,~~|\ell|\leq m.
\end{align}
Thus, by (\ref{u0}) and (\ref{P}), for $|\ell|\leq m, $
\begin{align*}
\left|\partial_\xi^\ell P_0\right|_{D(s_0,r_0)}&=\varepsilon\left|\partial_\xi^\ell P\right|_{D(s_0,r_0)}\leq \varepsilon^{\frac{1}{4}+\frac{5}{8}+\frac{1}{8(m+1)}}\varepsilon^{\frac{m}{8(m+1)}}\left|\partial_\xi^\ell P\right|_{D(s_0,r_0)}\\
&\leq s^m\varepsilon^{\frac{1}{4}+\frac{5}{8}+\frac{1}{8(m+1)}}\leq\gamma_0^{(m+1)(2d)^m}s_0^m\mu_0,
\end{align*}
which implies (\ref{P0}).

This completes the proof.
\end{proof}
\subsection{Induction from the $\nu$-th KAM step}

\subsubsection{Description of the $\nu$-th KAM step}
We first define the $\nu$-th KAM step parameters:
$$r_\nu=\frac{r_{\nu-1}}{2}+\frac{r_0}{4},~~~s_\nu=\frac{1}{8}s_{\nu-1}^{2\rho+1} ,~~~\mu_\nu=8^ms_{\nu-1}^\rho\mu_{\nu-1}.$$

Now, suppose that at $\nu$-th step, we have arrived at the following real analytic Hamiltonian:
\begin{equation}\label{eq2}
\begin{aligned}
H_\nu&=N_\nu+P_\nu
\end{aligned}
\end{equation}
with
\begin{align}\label{Nnu}
N_\nu&=e_\nu(\xi)+\langle\omega_\nu(\xi),y\rangle+\tilde h_\nu(y,\xi)+g_\nu(z,\xi)+\bar{g}_\nu(y,z,\xi),\\
\tilde h_\nu(y,\xi)&=\sum_{2\leq|i|}\tilde h_{i0}(\xi)y^i,\notag\\
g_\nu(z,\xi)&=g(z)+\sum_{i=0}^{\nu-1}\gamma_i^{(m+1)(2d)^m}s_i^{m-2}\mu_iO(|z|^2),\notag\\\label{bgnu}
\bar g_\nu(y,z,\xi)&=\sum_{2|i|+|j|\leq m,1\leq |i|, |j| }\bar g_{ij}(\xi)y^iz^j,
\end{align}
defined on $D(s_\nu,r_\nu)$,  
$$\nabla g_\nu(0)=\nabla g_{\nu-1}(\zeta_\nu-\zeta_{\nu-1})+\nabla_z[R_{\nu-1}](0,\zeta_\nu-\zeta_{\nu-1})=0,~~~~~\zeta_\nu\in B_{ (s_{\nu-1}^{m-2}\mu_{\nu-2})^\frac{1}{L}}(\zeta_{\nu-1}),\notag\\\label{degg}
$$
and
\begin{equation}\label{Pnu}
\left|\partial_\xi^\ell P_\nu\right|_{D(s_\nu,r_\nu)}\leq\gamma_\nu^{(m+1)(2d)^m}s_\nu^{m}\mu_\nu.
\end{equation}

For simplicity, {we will omit the index for all quantities of the present KAM step (at $\nu$-th step), use $+$ to index all quantities (Hamiltonians, domains, normal forms, perturbations, transformations, etc.) in the next KAM step (at ($\nu$+1)-th step), and use $-$ to index all quantities in the previous KAM step (at ($\nu$-1)-th step).}
To simplify notations, we will not specify the dependence of $P$, $P_+$ etc. All the constants $c_1$-$c_6$ (they will be defined in different lemmas-Lemma \ref{le1}, \ref{le2}, \ref{le5}, \ref{le4}, \ref{le7}, \ref{le8}) are positive and independent of the iteration process, and we will also use $c$ to denote any intermediate positive constant which is independent of the iteration process.
Define
\begin{align*}
r_+&=\frac{r}{2}+\frac{r_0}{4},\\
\beta_+&=\frac{\beta}{2}+\frac{\beta_0}{4},\\
\sigma_+&=\frac{\sigma}{2}+\frac{\sigma_0}{4},\\
\gamma_+&=\frac{\gamma}{2}+\frac{\gamma_0}{4},\\
s_+&=\frac{1}{8}\alpha s,~~~~~\alpha=s^{2\rho}=s^{\frac{1}{(m+1)}},\\
\mu_+&=8^mc_0\mu s^{\rho},~~~~~c_0=\max\{1,c_1,c_2,\cdots,c_6\},\\
K_+&=\left(\left[\log\frac{1}{s}\right]+1\right)^{3\eta},~~~~~(1+\rho)^\eta>2,\\
D(s)&=\{(y,z)\in \mathbb{C}^n\times \mathbb{R}^{2d}:|y|<s^2,|z|<s\},\\
\hat D(s)&=D\left(s,r_++\frac{7}{8}(r-r_+)\right),\\
\tilde{D}&=D\left(\beta_+,r_++\frac{5}{8}(r-r_+)\right),\\
D_{\frac{i}{8}\alpha}&=D\left(\frac{i}{8}\alpha s,r_++\frac{i-1}{8}(r-r_+)\right), ~~i=1,2,\cdots,8,\\
D_+&=D_{\frac{1}{8}\alpha}=D(s_+,r_+),\\
\Gamma(r-r_+)&=\sum_{0<|k|\leq K_+}|k|^{(m+1)(2d)^m\tau+m}e^{-|k|\frac{r-r_+}{8}},\\
G_{+}&=\Bigg\{\xi\in G: |\langle k,\omega\rangle|>\frac{\gamma}{|k|^\tau},\overline{A_{\imath\jmath}^\top} A_{\imath\jmath}>\frac{\gamma^2}{|k|^{2\tau}}I_{(2d)^{\imath+\jmath}},  for~ k\in\mathbb{Z}^n,\\
&~~~~\quad0<|k|<K_+, 2|\imath|+|\jmath|\leq m, 1\leq|\jmath|\Bigg\}.
\end{align*}

\subsubsection{Truncation}
Consider the Taylor-Fourier series of $P$:
\begin{equation*}
P=\sum_{k\in \mathbb{Z}^n,~\imath,\jmath\in \mathbb{Z}_+^n}p_{k\imath\jmath}y^{\imath}z^\jmath e^{\sqrt{-1}\langle k,x\rangle},
\end{equation*}
and let $R$ be the truncation of $P$ of the form
\begin{align*}
R&=\sum_{|k|\leq K_+,~2|\imath|+|\jmath|\leq m}p_{k\imath\jmath}y^{\imath}z^\jmath e^{\sqrt{-1}\langle k,x\rangle}\\
&=\sum_{|k|\leq K_+,~2|\imath|+|\jmath|\leq m}p_{k\imath\jmath}y_1^{\imath_1}\cdots y_n^{\imath_n}z_1^{\jmath_1}\cdots z_{2d}^{\jmath_{2d}}e^{\sqrt{-1}\langle k,x\rangle},
\end{align*}
where $|\imath|=|\imath_1|+\cdots+|\imath_n|$, $ |\jmath|=|\jmath_1|+\cdots+|\jmath_{2d}|$.

Next, we will prove that the residual term $P-R$ is much smaller than perturbation term $P$ by truncating appropriately, see the lemma below.
\begin{lemma}\label{le1}
Assume that
$${\textsc{(H1)}}:\int_{K_+}^{\infty}t^{n}e^{-t\frac{r-r_+}{16}}dt\leq s.$$
Then there is a constant $c_1$ such that for all $\xi\in G$, $|\ell|\leq m$,
\begin{align}\label{P-R}
\left|\partial_\xi^\ell (P-R)\right|_{D_\alpha}&\leq c_1\gamma^{(m+1)(2d)^m}s^{m+1}\mu,\\\label{R}
\left|\partial_\xi^\ell R\right|_{D_\alpha}&\leq c_1\gamma^{(m+1)(2d)^m}s^{m}\mu.
\end{align}
\end{lemma}
\begin{proof}
Denote
\begin{align*}
I&=\sum_{|k|>K_+,~\imath,\jmath\in \mathbb{Z}_+^n}p_{k\imath\jmath}y^{\imath}z^\jmath e^{\sqrt{-1}\langle k,x\rangle},\\
II&=\sum_{|k|\leq K_+,~2|\imath|+|\jmath|>m}p_{k\imath\jmath}y^{\imath}z^\jmath e^{\sqrt{-1}\langle k,x\rangle}.
\end{align*}
Then
\begin{align*}
P-R=I+II.
\end{align*}
To estimate $I$, we notice by (\ref{Pnu}) that
\begin{align*}
\left|\sum_{\imath\in \mathbb{Z}_+^n}\partial_\xi^\ell p_{k\imath}y^{\imath}\right|\leq \left|\partial_\xi^\ell P\right|_{D(s,r)}e^{-|k|r}\leq\gamma^{(m+1)(2d)^m}s^{m}\mu e^{-|k|r},
\end{align*}
where the first inequality has been frequently used in \cite{han,li,poschel}. This together with ${\textsc{(H1)}}$ and (\ref{Pnu}) yield
\begin{align}\label{I}
\left|\partial_\xi^\ell I\right|_{\hat D(s)}&\leq\sum_{|k|>K_+}\left|\partial_\xi^\ell P\right|_{D(s,r)}e^{-|k|\frac{r-r_+}{8}}\leq\gamma^{(m+1)(2d)^m}s^{m}\mu\sum_{\kappa=K_+}^{\infty}\kappa^{n}e^{-\kappa\frac{r-r_+}{8}}\notag\\
&\leq\gamma^{(m+1)(2d)^m}s^{m}\mu\int_{K_+}^{\infty}t^{n}e^{-t\frac{r-r_+}{16}}dt\leq\gamma^{(m+1)(2d)^m}s^{m+1}\mu.
\end{align}
It follows from (\ref{Pnu}) and (\ref{I}) that
\begin{align*}
\left|\partial_\xi^\ell (P-I)\right|_{\hat D(s)}\leq\left|\partial_\xi^\ell P\right|_{D(s,r)}+\left|\partial_\xi^\ell I\right|_{\hat D(s)}\leq2\gamma^{(m+1)(2d)^m}s^{m}\mu.
\end{align*}
For $2|p|+|q|=m+1$, let $\int$ be the obvious antiderivative of $\frac{\partial^{(p,q)}}{\partial y^pz^q}$. Then the Cauchy estimate of $P-I$ on $D_*$ yields
\begin{align*}
\left|\partial_\xi^\ell II\right|_{D_\alpha}&=\left|\partial_\xi^\ell \int\frac{\partial^{(p,q)}}{\partial y^pz^q}\sum_{|k|\leq K_+,~2|\imath|+|\jmath|> m} p_{k\imath\jmath}y^{\imath}z^\jmath e^{\sqrt{-1}\langle k,x\rangle}dydz\right|_{D_{\alpha}}\\
&\leq\left|\int\left|\frac{\partial^{(p,q)}}{\partial y^pz^q}\partial_\xi^\ell (P-I)\right|dydz\right|_{D_{\alpha}}\\
&\leq\left|\frac{c}{s^{m+1}}\int\left|\partial_\xi^\ell (P-I)\right|_{D_*}dydz\right|_{D_{\alpha}}\\
&\leq2\frac{c}{s^{m+1}}\gamma^{(m+1)(2d)^m}s^{m}\mu (\alpha s)^{m+1}\\
&\leq c\gamma^{(m+1)(2d)^m}s^{m+1}\mu.
\end{align*}
Thus,
\begin{align*}
\left|\partial_\xi^\ell (P-R)\right|_{D_\alpha}\leq c\gamma^{(m+1)(2d)^m}s^{m+1}\mu,
\end{align*}
and therefore,
\begin{align*}
\left|\partial_\xi^\ell R\right|_{D_\alpha}\leq\left|\partial_\xi^\ell(P-R)\right|_{D_\alpha}+\left|\partial_\xi^\ell P\right|_{D(s,r)}\leq c\gamma^{(m+1)(2d)^m}s^{m}\mu.
\end{align*}

This completes the proof.
\end{proof}

\subsubsection{Construct a symplectic transformation}

We will construct a symplectic transformation $\Phi_{+}$:
\begin{align*}\label{Phi+}
\Phi_{+}: D(s_{+},r_{+})\times G_+\rightarrow  D(s,r)\times G
\end{align*}
such that it transforms Hamiltonian ($\ref{eq2}$) into the Hamiltonian of the next KAM cycle (at ($\nu$+1)-th step), i.e.,
\begin{equation*}\label{H+}
H_{+}=H\circ\Phi_{+}=N_{+}+P_{+},
\end{equation*}
where $N_{+}$ and $P_{+}$ have similar properties as $N$ and $P$ respectively on $D(s_{\nu+1},r_{\nu+1})$. The construction contains two steps: average process (see subsection \ref{HomologicalEquation}) and translation (see subsection \ref{translation}).

\subsubsection{Homological equation}\label{HomologicalEquation}
As usual, we shall construct a symplectic transformation as the time-1 map $\phi_{F}^1$ of the flow generated by a Hamiltonian $F$
of the form
\begin{align}\label{eq3}
F&=\sum_{0<|k|\leq K_+,~2|\imath|+|\jmath|\leq m}F_{k\imath\jmath}y^{\imath}z^\jmath e^{\sqrt{-1}\langle k,x\rangle}\notag\\
&=\sum_{0<|k|\leq K_+,~2|\imath|+|\jmath|\leq m}F_{k\imath\jmath}y_1^{\imath_1}\cdots y_n^{\imath_n}z_1^{\jmath_1}\cdots z_{2d}^{\jmath_{2d}}e^{\sqrt{-1}\langle k,x\rangle},
\end{align}
where $F_{k\imath\jmath}$ are $(y, z, \xi)$-dependent vectors or matrices of obvious dimension.

The Hamiltonian $F$ can be determined by the following quasi-linear homological equation
\begin{equation}\label{eq4}
\{N,F\}+R-[R]-Q=0,
\end{equation}
where $[R]=\frac{1}{(2\pi)^n}\int_{\mathbb{T}^n}R(x,y,z)dx$ is the average of the truncation $R$, and the correction term
\begin{align}\label{Q}
Q=(\partial_zg+\partial_z\bar g)J\partial_zF\big|_{2|\imath|+|\jmath|> m}, 
\end{align}
corresponds to the $(m+1)$-order and higher-order terms in $\langle \nabla_z \bar g, {\tilde{J}}\nabla_zF\rangle$.

{Recall (\ref{Nnu}), i.e., \begin{align*}
N=e(\xi)+\langle\omega(\xi),y\rangle+\tilde h(y,\xi)+g(z,\xi)+\bar{g}(y,z,\xi), \end{align*}
with
\begin{align*}
\tilde h(y,\xi)=\sum_{2\leq|i|}\tilde h_{0i0}(\xi)y^i,~~~
g(z,\xi)=\sum_{2\leq|j|}g_{00j}(\xi)z^j,~~~
\bar g(y,z,\xi)=\sum_{2|i|+|j|\leq m,1\leq |i|, |j| }\bar g_{0ij}(\xi)y^iz^j.
\end{align*}

Notice that
\begin{align}\label{NF}
\{N,F\}&=-\partial_yN\partial_xF+\partial_xN\partial_yF+\partial_zNJ\partial_zF\notag\\
&=-\sqrt{-1}\left\langle k,\omega+\nabla \tilde{h}(y)+\partial_y\bar g\right\rangle F_{k\imath\jmath}y^{\imath}z^\jmath e^{\sqrt{-1}\langle k,x\rangle}+(\partial_zg+\partial_z\bar g)J\partial_zF.
\end{align}
In order to simplify the notations, we sometimes omit the subscript of $\sum$ and only use $\sum$ to represent the sum to the index over the corresponding range.

Now we calculate the last term of (\ref{NF}):
\begin{align}\label{mBC}
(\partial_zg+\partial_z\bar g)J\partial_zF=(\partial_zg+\partial_z\bar g)J\partial_zF\big|_{2|\imath|+|\jmath|\leq m}+Q.
\end{align}
Specially,
\begin{align}\label{mB}
\partial_zgJ\partial_zF\big|_{2|\imath|+|\jmath|\leq m}
&=\partial_zgJ\sum_{0<|k|\leq K_+,~2|\imath|+|\jmath|\leq m} F_{k\imath\jmath}y^{\imath}\partial_z(z^{\jmath}) e^{\sqrt{-1}\langle k,x\rangle}\notag\\
&\quad+\partial_zgJ\sum_{0<|k|\leq K_+,~2|\imath|+|\jmath|\leq m} \partial_z(F_{k\imath\jmath})y^{\imath}z^{\jmath} e^{\sqrt{-1}\langle k,x\rangle}\notag\\
&=:\sum_{} S_{\imath\jmath}F_{k\imath\jmath}y^\imath z^{\jmath}e^{\sqrt{-1}\langle k,x\rangle}+\sum_{2\leq|j'|}S_{j'}\partial_zF_{k\imath(\jmath-j'+1)}y^\imath z^\jmath e^{\sqrt{-1}\langle k,x\rangle}\notag\\
&=:\sum_{}\tilde S_{\imath\jmath}F_{k\imath\jmath}y^\imath z^{\jmath}e^{\sqrt{-1}\langle k,x\rangle}+\sum_{}(S_{\imath\jmath}-\tilde S_{\imath\jmath})F_{k\imath\jmath}y^\imath z^{\jmath}e^{\sqrt{-1}\langle k,x\rangle}\\
&\quad+\sum_{2\leq|j'|}S_{j'}\partial_zF_{k\imath(\jmath-j'+1)}y^\imath z^\jmath e^{\sqrt{-1}\langle k,x\rangle},\notag
\end{align}
where $0<|k|\leq K_+$, $2|\imath|+|\jmath|\leq m$, $2\leq|j'|$, $S_{\imath\jmath}$ is a ($|\imath|+|\jmath|$) order tensor and concerned with $\partial_z^2g(z)$ and $J$, $S_{j'}$ is concerned with $\partial_z^{j'}g(0)$ and $J$,
$\tilde S_{\imath\jmath}$ is a ($|\imath|+|\jmath|$) order tensor and concerned with $\partial_z^2g(0)$ and $J$.
And \begin{align}\label{mC}
\partial_z\bar gJ\partial_zF\big|_{2|\imath|+|\jmath|\leq m}=&\partial_z\bar g
J
\sum_{0<|k|\leq K_+,~2|\imath|+|\jmath|\leq m} F_{k\imath\jmath}y^{\imath}\partial _z(z^{\jmath}) e^{\sqrt{-1}\langle k,x\rangle}\notag\\
&+\partial_z\bar g
J
\sum_{0<|k|\leq K_+,~2|\imath|+|\jmath|\leq m} \partial _z(F_{k\imath\jmath})y^{\imath}z^{\jmath} e^{\sqrt{-1}\langle k,x\rangle}\notag\\
=:&\sum S_{ij}\left(F_{k(\imath-i)(\jmath+2-j)}+\partial_zF_{k(\imath-i)(\jmath+1-j)}\right)y^\imath z^\jmath e^{\sqrt{-1}\langle k,x\rangle}
\end{align}
where $S_{ij}$ is concerned with $\partial_{y}^{i}\partial_z^{j}\bar g(0,0)$ and $J$, and $2|i|+|j|\leq m$, $1\leq|i|$, $1\leq |j|$, ${0<|k|\leq K_+,~2|\imath|+|\jmath|\leq m}$.

Substituting (\ref{mBC}), (\ref{mB}) and (\ref{mC}) into (\ref{NF}), putting (\ref{Q}) and (\ref{NF}) into (\ref{eq4}),
comparing the coefficients, we thus obtain the following quasi-linear equations for all $0<|k|\leq K_+$, $2|\imath|+|\jmath|\leq m$:
\begin{align}\label{eq5}
&\left(\sqrt{-1}\left\langle k,\omega+\nabla \tilde{h}(y)+\nabla_y\bar{g}\right\rangle I_{(2d)^{\imath+\jmath}} -S_{\imath\jmath}\right)F_{k\imath\jmath}=p_{k\imath\jmath}+S_{j'}\partial_zF_{k\imath(\jmath-j'+1)}\\
&+S_{ij}F_{k(\imath-i)(\jmath+2-j)}+S_{ij}\partial_zF_{k(\imath-i)(\jmath+1-j)},\notag
\end{align}
where $i, j, j'$ are defined as above,  $S_{ij}F_{k(\imath-i)(\jmath+2-j)}$ stands for
$\sum_{2|i|+|j|\leq m, 1\leq |i|,|j|}S_{ij}F_{k(\imath-i)(\jmath+2-j)}$, and the rest terms are analogously defined.}

The above equations (\ref{eq5}) are solvable if the coefficient matrices are
nonsingular. We denote
\begin{align}\label{O+}
{G_{+}}=&\left\{\xi\in G: |\langle k,\omega\rangle|>\frac{\gamma}{|k|^\tau},\overline{A_{\imath\jmath}^\top} A_{\imath\jmath}>\frac{\gamma^2}{|k|^{2\tau}}I_{(2d)^{\imath+\jmath}},  for~ k\in\mathbb{Z}^n,~0<|k|<K_+, 2|\imath|+|\jmath|\leq m, 1\leq|\jmath|\right\},
\end{align}
with
\begin{align}\label{A}
A_{\imath\jmath}=\sqrt{-1}\left\langle \frac{k}{|k|},\omega\right\rangle I_{(2d)^{\imath+\jmath}}+\frac{\tilde{S}_{\imath\jmath}}{|k|},
\end{align}
where $\tilde{S}_{\imath\jmath}$ is concerned with $\frac{\partial^2g(0)}{\partial z^2}$ and ${J}$.
Then, we can solve equations (\ref{eq5}) on $G_+$.
The details can be seen in the following lemma:
\begin{lemma}\label{le2}
Assume that
\begin{align*}
&{\textsc{(H2)}}:\max_{|\ell|\leq m,|i|\leq m}\left|\partial_\xi^\ell\partial_y^i\tilde{h}- \partial_\xi^\ell\partial_y^i\tilde{h}_0\right|_{D(s)\times G_+}\leq s_0^{\frac{1}{2}},\\
&{\textsc{(H3)}}:4s<\frac{\gamma-\gamma_+}{(M^*+2)K_+^{\tau+1}},
\end{align*}
where
\begin{align*}
M^*=\max_{|\ell|\leq m,|i|\leq m }\left| \partial_\xi^\ell\partial_{y}^i\tilde h_0(y)\right|_{D(s)\times G_+}.
\end{align*}
The quasi-linear equations (\ref{eq5}) can be uniquely solved on $D(s)\times G_+$ to obtain the coefficient $F_{k\imath\jmath}$ which satisfy the following equality:
\begin{equation*}
\left|\partial_\xi^\ell\partial_y^i\partial_z^jF_{k\imath\jmath}\right|_{D(s)\times G_+}\leq c_2|k|^{|\ell|+|i|+|j|+(|\ell|+|i|+|j|+1)(2d)^\jmath\tau}s^{m-2|\imath|-|\jmath|}\mu e^{-|k|r},
\end{equation*}
for all $0<|k|\leq K_+$, $2|\imath|+|\jmath|\leq m$, $|\ell|+|i|+|j|\leq m$, where $c_2$ is a constant.
\end{lemma}
\begin{proof}
For $\forall (y,\xi)\in D(s)\times G_+$, by {(H2)},{(H3)},
\begin{align}\label{eq11}
&\left|\nabla \tilde{h}\right|=\left|(\nabla \tilde{h}-\nabla \tilde{h}_0)+\nabla \tilde{h}_0\right|
\leq(1+M^*)|y|
<(1+M^*)s
<\frac{\gamma}{4|k|^{\tau+1}}.
\end{align}
Recalling that $\bar g(y,z)$ comes from the perturbation and it is formulized in (\ref{bgnu}), we have
\begin{align}\label{ybarg}
|\nabla_y\bar{g}|\leq cs<\frac{\gamma}{4|k|^{\tau+1}},
\end{align}
with $c\leq\sum_{i=0}^{\nu-1}\gamma_{i}^{(m+1)(2d)^m}s_{i}^{m-3}\mu_{i}\ll1$, where the last equality follows from {(H3)}.
Notice by (\ref{eq11}) and (\ref{ybarg}) that
\begin{align*}
\left|\det\left\langle k,\nabla \tilde{h}+\nabla_y\bar{g}\right\rangle I_{(2d)^{\imath+\jmath}}\right|\leq \left(\frac{\gamma}{4|k|^\tau}\right)^{(2d)^{\imath+\jmath}},~~2|\imath|+|\jmath|\leq m.
\end{align*}
It follows from the definition of $S_{\imath\jmath}$, $\tilde S_{\imath\jmath}$ and $g$ that \begin{align*}
  \left|\det \left(S_{\imath\jmath}-\tilde{S}_{\imath\jmath}\right)\right|\leq\left|\det\left(|s|I_{(2d)^{\imath+\jmath}}\right)\right|<\left(\frac{\gamma}{4|k|^\tau}\right)^{(2d)^{\imath+\jmath}},~~2|\imath|+|\jmath|\leq m.
\end{align*}
This together with (\ref{eq5}) and (\ref{O+}) yield
\begin{align}\label{Bj}
\left|\det B_{\imath\jmath}\right|> \frac{\gamma^{(2d)^{\imath+\jmath}}}{2\left(|k|^\tau\right)^{(2d)^{\imath+\jmath}}},~~2|\imath|+|\jmath|\leq m,
\end{align}
where $B_{\imath\jmath}$ is the coefficient matrix of (\ref{eq5}). Then, by (\ref{Bj}),
\begin{align*}
\left|B_{\imath\jmath}\right|^{-1}=\left|\frac{\textrm{adj} B_{\imath\jmath}}{\det B_{\imath\jmath}}\right|\leq c\frac{|k|^{\tau(2d)^\jmath+(2d)^\jmath-1}}{\gamma^{(2d)^\jmath}},~~2|\imath|+|\jmath|\leq m,
\end{align*}
where $\textrm{adj} B_{\imath\jmath}$ is the algebraic minor of $B_{\imath\jmath}$.
Applying the identity
\begin{align*}
\partial_z^jB_{\imath\jmath}^{-1}=-\sum_{|j'|=1}^{|j|}\left(\begin{array}{c}
j\\
j'
\end{array}\right)\left(\partial_z^{j-j'}B_{\imath\jmath}^{-1}\partial_z^{j'}B_{\imath\jmath}\right)B_{\imath\jmath}^{-1}
\end{align*}
inductively, we have for $|\ell|+|i|+|j|\leq m$
\begin{align*}
\left|\partial_\xi^\ell\partial_y^i\partial_z^j B_{\imath\jmath}^{-1}\right|_{ D(s)\times G_+}&\leq c|k|^{|\ell|+|i|+|j|}|B_{\imath\jmath}^{-1}|^{|\ell|+|i|+|j|+1}\\
&\leq c\frac{|k|^{|\ell|+|i|+|j|+(|\ell|+|i|+|j|+1)(2d)^\jmath\tau}}{\gamma^{(|\ell|+|i|+|j|+1)(2d)^\jmath}},~~2|\imath|+|\jmath|\leq m,
\end{align*}
which was precisely proved in \cite{li2,li}. We note by the Cauchy estimate that for $2|\imath|+|\jmath|\leq m$,
\begin{align*}
\left|\partial_\xi^\ell p_{k\imath\jmath}\right|_{G_+}&\leq\left| \partial_\xi^\ell P\right|_{ D(s)\times G_+}s^{-2|\imath|-|\jmath|}e^{-|k|r}\leq\gamma^{(m+1)(2d)^m}s^{m-2|\imath|-|\jmath|}\mu e^{-|k|r}.
\end{align*}
So, for $|\ell|+|i|+|j|\leq m$, $2|\imath|+|\jmath|\leq m$,
\begin{align*}
|\partial_\xi^\ell\partial_y^i\partial_z^jF_{k\imath\jmath}|_{ D(s)\times G_+}&\leq c\frac{|k|^{|\ell|+|i|+|j|+(|\ell|+|i|+|j|+1)(2d)^\jmath\tau}}{\gamma^{(|\ell|+|i|+|j|+1)(2d)^\jmath}}
\gamma^{(m+1)(2d)^m}s^{m-2|\imath|-|\jmath|}\mu e^{-|k|r}\\
&\leq c|k|^{|\ell|+|i|+|j|+(|\ell|+|i|+|j|+1)(2d)^\jmath\tau}s^{m-2|\imath|-|\jmath|}\mu e^{-|k|r}.
\end{align*}
This completes the proof.
\end{proof}
Next, we apply the above transformation $\phi_F^1$ to Hamiltonian $H$, i.e.,
\begin{align*}
H\circ\phi_F^1&=(N+R)\circ\phi_F^1+(P-R)\circ\phi_F^1\\
&=(N+R)+\{N,F\}+\int_0^1\{(1-t)\{N,F\}+ R,F\}\circ\phi_F^tdt+(P-R)\circ\phi_F^1\\
&=N+[R]+\int_0^1\{R_t,F\}\circ\phi_F^tdt+(P-R)\circ\phi_F^1+Q\\
&=:\bar N_++\bar P_+,
\end{align*}
where
\begin{align}
& \bar N_+=N+[R],\notag\\\label{eq16}
& \bar P_+=\int_0^1\{R_t,F\}\circ\phi_F^tdt+(P-R)\circ\phi_F^1+Q,\\
&R_t=(1-t)Q+(1-t)[R]+tR.\notag
\end{align}
\subsubsection{Translation}\label{translation}
In this subsection, we will construct a translation so as to eliminate the first order terms about $z$.

Consider the translation
$$\phi:x\rightarrow x,~~~~~y\rightarrow y,~~~~~z\rightarrow z+\zeta_+-\zeta,$$
where $z=(u,v)$, $\zeta_+$ is to be determined.
Let
$$\Phi_+=\phi_F^1\circ\phi.$$
Then
\begin{align}
H\circ\Phi_+&=N_++P_+,\notag\\
N_+&=\bar N_+\circ\phi,\notag\\\label{eq18}
P_+&=\bar P_+\circ\phi
\end{align}
with
\begin{align*}
N_+&=\bar{N}_+\circ\phi=(N+[R])\circ\phi=(e+\langle\omega,y\rangle+\tilde h(y)+g(z)+\bar{g}(y, z)+[R](y,z))\circ\phi\\
&=e+\langle\omega,y\rangle+\tilde h(y)+g(z+\zeta_+-\zeta)+\bar{g}(y, z+\zeta_+-\zeta)+[R](y,z+\zeta_+-\zeta)\\
&=:e_++\langle\omega_+,y\rangle+\tilde h_++g_++\bar{g}_+,
\end{align*}
where
\begin{align}\label{e+}
e_+&=e+g(\zeta_+-\zeta)+[R](0,\zeta_+-\zeta),\\\label{omega+}
\omega_+&=\omega+\nabla_y[R](0,\zeta_+-\zeta)+\nabla_y\bar g(0,\zeta_+-\zeta),\\\label{h+}
\tilde h_+&=\tilde h(y)+[R](y,\zeta_+-\zeta)-[R](0,\zeta_+-\zeta)-\langle\nabla_y[R](0,\zeta_+-\zeta),y\rangle\\
&\quad+\bar g(y,\zeta_+-\zeta)-\langle\nabla_y\bar g(0,\zeta_+-\zeta),y\rangle,\notag\\\label{g+}
g_+&=g(z+\zeta_+-\zeta)-g(\zeta_+-\zeta)+[R](0,z+\zeta_+-\zeta)-[R](0,\zeta_+-\zeta),\\\label{barg}
\bar{g}_+&=\bar{g}(y, z+\zeta_+-\zeta)-\bar g(y,\zeta_+-\zeta)+[R](y,z+\zeta_+-\zeta)-[R](y,\zeta_+-\zeta)\\
&\quad-[R](0,z+\zeta_+-\zeta)+[R](0,\zeta_+-\zeta).\notag
\end{align}

\subsubsection{Eliminate the first order terms about $z$ }
In this subsection, we will appropriately choose $\zeta_+$ to remove the first order terms about $z$. The concrete details are expressed as the following lemma, which is crucial to our proof.
\begin{lemma}\label{le3}
Notice that
\begin{align}\label{eq47}
\nabla g_+(0)=\nabla g(\zeta_+-\zeta)+\nabla_z[R](0,\zeta_+-\zeta).
\end{align}
There exists
$\zeta_+\in B_{ (s_{-}^{m-1}\mu_{-})^\frac{1}{L}}(\zeta)$
such that
\begin{align*}
\nabla g_+(0)=\nabla g(0)=\cdots=\nabla g_0(0)=0.
\end{align*}
\end{lemma}
\begin{proof}
The proof will be completed by induction on $\nu$.
We start with the case $\nu=0$. It follows from (\ref{eq62}) and {(A0)} that
\begin{align*}
&\nabla g_0(\zeta_0)=\nabla g_0(0)=0,\\
&\deg(\nabla g_0(\cdot)-\nabla g_0(0),\mathcal{O}^o,0)\neq0,\\
&\left|\nabla g_0(z)-\nabla g_0(z_*)\right|\geq\sigma_0|z-z_*|^L.
\end{align*}
Now assume that for some $\nu\geq1$ we have got
\begin{align}
&\nabla g_i(0)=\nabla g_{i-1}(\zeta_i-\zeta_{i-1})+\nabla_z[R_{i-1}](0,\zeta_i-\zeta_{i-1})=0,\notag\\\label{degg}
&\deg(\nabla g_i(\cdot)-\nabla g_i(0),\mathcal{O}^o,0)\neq0,\\\label{nag}
&|\nabla g_i(z)-\nabla g_i(z_*)|\geq\sigma_i|z-z_*|^L,
\end{align}
where $i=1,2,\cdots,\nu,$ $\zeta_i\in B_{ (s_{i-2}^{m-1}\mu_{i-2})^\frac{1}{L}}(\zeta_{i-1})$, $s_{-1}=s_0$, $\mu_{-1}=\mu_0$, $z_*\in\mathcal{O}$, $z\in \mathcal{O}\backslash B_{(s_{i-1}^{m-1}\mu_{i-1})^{\frac{1}{L}}}(z_*). $
Then, we need to find $\zeta_+$ near $\zeta$ such that $\nabla g_+(0)=\nabla g(0)$.

Consider homotopy $\mathcal{H}_t(z):[0,1]\times \mathcal{O}\rightarrow \mathbb{R}^{2d}$,
\begin{align*}
\mathcal{H}_t(z)&=:\nabla g(z-\zeta)-\nabla g(0)+t\nabla_z[R](0,z-\zeta).
\end{align*}
Notice that
\begin{align}\label{ezr}
&\left|\nabla_z[R](y,z)\right|\leq\gamma^{(m+1)(2d)^m}s^{m-1}\mu.
\end{align}
For any $z\in\partial\mathcal{O}$, $t\in[0,1]$, by (\ref{nag}) and (\ref{ezr}), we have
\begin{align*}
\left|\mathcal{H}_t(z)\right|
&\geq\left|\nabla g(z-\zeta)-\nabla g(0)|-|\nabla_z[R](0,z-\zeta)\right|\\
&\geq\sigma|z-\zeta|^L-\gamma^{(m+1)(2d)^m}s^{m-1}\mu\\
&>\frac{\sigma\delta^L}{2},
\end{align*}
where $\delta:=\min\left\{|z-\zeta|, \forall z\in\partial\mathcal{O}\right\}$.
So, it follows from the homotopy invariance and (\ref{degg}) that
\begin{align}\label{H1}
\deg(\mathcal{H}_1(\cdot),\mathcal{O}^o,0)=\deg(\mathcal{H}_0(\cdot),\mathcal{O}^o,0)\neq0.
\end{align}
%
We note by (\ref{nag}) and (\ref{ezr}) that for any $z\in\mathcal{O}\backslash B_{(s_{-}^{m-1}\mu_{-})^{\frac{1}{L}}}(\zeta)$,
\begin{align*}
|\mathcal{H}_1(z)|&=|\nabla g(z-\zeta)-\nabla g(0)+\nabla_z[R](0,z-\zeta)|\\
&\geq|\nabla g(z-\zeta)-\nabla g(0)|-|\nabla_z[R](0,z-\zeta)|\\
&\geq\sigma|z-\zeta|^L-\gamma^{(m+1)(2d)^m}s^{m-1}\mu\\
&\geq\sigma s_{-}^{m-1}\mu_{-}-\gamma^{(m+1)(2d)^m}s^{m-1}\mu\\
&\geq\frac{\sigma}{2}s_{-}^{m-1}\mu_{-}.
\end{align*}
Hence, by excision and (\ref{H1}),
\begin{align*}
\deg(\mathcal{H}_1(\cdot),B_{ (s_{-}^{m-1}\mu_{-})^{\frac{1}{L}}}(\zeta),0)=\deg(\mathcal{H}_1(\cdot),\mathcal{O}^o,0)\neq0,
\end{align*}
then there exist at least a $\zeta_+\in B_{ (s_{-}^{m-1}\mu_{-})^{\frac{1}{L}}}(\zeta)$, such that
\begin{align*}
\mathcal{H}_1(\zeta_+)=0,
\end{align*}
i.e.,
\begin{align*}
\nabla g(\zeta_+-\zeta)+\nabla_z[R](0,\zeta_+-\zeta)=\nabla g(0),
\end{align*}
thus, by (\ref{eq47}),
\begin{align}\label{g+=g0}
\nabla g_+(0)=\nabla g(0)=\cdots=\nabla g_0(0)=0.
\end{align}
Next, we need to prove
\begin{align}\label{degg+}
&\deg(\nabla g_+(\cdot)-\nabla g_+(0),\mathcal{O}^o,0)\neq0,\\\label{nag+}
&|\nabla g_+(z)-\nabla g_+(z_*)|\geq\sigma_+|z-z_*|^L.
\end{align}
By (\ref{g+}),
\begin{align*}
\nabla g_+(z)=\nabla g(z+\zeta_+-\zeta)+\nabla_z [R](0,z+\zeta_+-\zeta).
\end{align*}
Then
\begin{align}\label{g+-g}
\nabla g_+(z)-\nabla g(z)&=\nabla g(z+\zeta_+-\zeta)-\nabla g(z)+\nabla_z [R](0,z+\zeta_+-\zeta),
\end{align}
and
\begin{align}\label{g+-g+}
\nabla g_+(z)-\nabla g_+(z_*)&=\nabla g(z+\zeta_+-\zeta)-\nabla g(z_*+\zeta_+-\zeta)\\
&\quad+\nabla_z [R](0,z+\zeta_+-\zeta)-\nabla_z [R](0,z_*+\zeta_+-\zeta).\notag
\end{align}
In view of (\ref{ezr}), (\ref{g+-g}), and $\zeta_+\in B_{(s_{-}^{m-1}\mu_{-})^\frac{1}{L}}(\zeta)$, we get
\begin{align*}
|\nabla g_+(z)-\nabla g(z)|\leq c(s_{-}^{m-1}\mu_{-})^\frac{1}{L}.
\end{align*}
This together with the property of degree, (\ref{degg}) and (\ref{g+=g0}) that (\ref{degg+}) holds, i.e.,
\begin{align*}
\deg(\nabla g_+(\cdot)-\nabla g_+(0),\mathcal{O}^o,0)&=\deg(\nabla g_+(\cdot)-\nabla g(\cdot)+\nabla g(\cdot)-\nabla g(0),\mathcal{O}^o,0)\\
&=\deg(\nabla g(\cdot)-\nabla g(0),\mathcal{O}^o,0)\neq0.
\end{align*}
It follows from (\ref{nag}), (\ref{ezr}) and (\ref{g+-g+}) that for any $z\in\mathcal{O}\backslash B_{(s^{m-1}\mu)^{\frac{1}{L}}}(z_*)$
\begin{align*}
|\nabla g_+(z)-\nabla g_+(z_*)|\geq \sigma|z-z_*|^L-2\gamma^{(m+1)(2d)^m}s^{m-1}\mu\geq \sigma_+|z-z_*|^L,
\end{align*}
which implies (\ref{nag+}).

This completes the proof.
\end{proof}
\subsubsection{Estimate on $N_+$}
Now, we give the estimate of $N_+$.
\begin{lemma}\label{le5}
There is a constant $c_3$ such that for all $|\ell|\leq m$:
\begin{align}\label{eq23}
&|\partial_\xi^\ell(\zeta_{+}-\zeta)|_{G_+}\leq c_3 (s_{-}^{m-1}\mu_{-})^{\frac{1}{L}},\\\label{eq24}
&|\partial_\xi^\ell (e_{+}- e)|_{G_+}\leq c_3(s_{-}^{m-1}\mu_{-})^{\frac{1}{L}},\\\label{omega+-omega}
&|\partial_\xi^\ell(\omega_+-\omega)|_{G_+}\leq c_3(s_{-}^{m-1}\mu_{-})^{\frac{1}{L}},\\\label{eq25}
&|\partial_\xi^\ell(\tilde h_+-\tilde h)|_{D(s_+)\times G_+}\leq c_3(s_{-}^{m-1}\mu_{-})^{\frac{1}{L}},\\\label{eq48}
&|\partial_\xi^\ell( g_+- g)|_{D(s_+)\times G_+}\leq c_3(s_{-}^{m-1}\mu_{-})^{\frac{1}{L}},\\\label{bg+-bg}
&|\partial_\xi^\ell(\bar{g}_+-\bar{g})|_{D(s_+)\times G_+}\leq c_3(s_{-}^{m-1}\mu_{-})^{\frac{1}{L}}.
\end{align}
\end{lemma}
\begin{proof}
Differentiating (\ref{eq47}) with respect to $\xi$ yields $$\left(\nabla^2g(\zeta_+-\zeta)+\nabla_z^2[R](0,\zeta_+-\zeta)\right)\partial_\xi(\zeta_+-\zeta)=\partial_\xi\nabla g(\zeta_+-\zeta)+\partial_\xi\nabla_z[R](0,\zeta_+-\zeta).$$
Applying this identity inductively and using  $\zeta_+\in B_{ (s_{-}^{m-1}\mu_{-})^{\frac{1}{L}}}(\zeta)$ in Lemma \ref{le3}, we can immediately get (\ref{eq23}).
Recall $e_+$ and $\omega_+$ in (\ref{e+}) and (\ref{omega+}), respectively, we see that
\begin{align*}
\left|\partial_\xi^\ell(e_+-e)\right|_{G_+}&\leq c\left|\partial_\xi^\ell(\zeta_+-\zeta)\right|^2+\gamma^{(m+1)(2d)^m}s^m\mu\leq c\left(s_{-}^{m-1}\mu_{-}\right)^{\frac{1}{L}},\\
\left|\partial_\xi^\ell(\omega_+-\omega)\right|_{G_+}&\leq c\left|\partial_\xi^\ell(\zeta_+-\zeta)\right|+  \gamma^{(m+1)(2d)^m}s^{m-2}\mu\leq c\left(s_{-}^{m-1}\mu_{-}\right)^{\frac{1}{L}},
\end{align*}
which implies (\ref{eq24}) and (\ref{omega+-omega}).

In view of (\ref{h+}), we have that
\begin{align*}
\left|\partial_\xi^\ell(\tilde h_+(y)-\tilde h(y))\right|_{D(s_+)}&=\left|\partial_\xi^\ell([R](y,\zeta_+-\zeta)-[R](0,\zeta_+-\zeta)-\langle\nabla_y[R](0,\zeta_+-\zeta),y\rangle)\right|\\
&~~~~+\left|\partial_\xi^\ell(\bar g(y^\imath(\zeta_+-\zeta)^\jmath)-\langle\nabla_y\bar g((\zeta_+-\zeta)^\jmath),y\rangle)\right|\\
&\leq \gamma^{(m+1)(2d)^m}s^{m}\mu+\left(s_{-}^{m-1}\mu_{-}\right)^{\frac{1}{L}}s_+^4 \\
&\leq c\left(s_{-}^{m-1}\mu_{-}\right)^\frac{1}{L}.
\end{align*}

It follows from (\ref{g+}) that
\begin{align*}
\left|\partial_\xi^\ell(g_+(z)-g(z))\right|_{D(s_+)}&=|\partial_\xi^\ell(g(z+\zeta_+-\zeta)-g(\zeta_+-\zeta)-g(z))|\\
&~~~~+\left|\partial_\xi^\ell([R](0,z+\zeta_+-\zeta)-[R](0,\zeta_+-\zeta))\right|\\
&\leq c\left(s_{-}^{m-1}\mu_{-}\right)^\frac{1}{L}s_+^2+\gamma^{(m+1)(2d)^m}s^{m}\mu\\
&\leq c\left(s_{-}^{m-1}\mu_{-}\right)^\frac{1}{L}.
\end{align*}

According to (\ref{barg}), we obtain that
\begin{align*}
\left|\partial_\xi^\ell(\bar{g}_+(y^\imath z^\jmath)-\bar{g}(y^\imath z^\jmath))\right|_{D(s_+)}\leq\left|\partial_\xi^\ell[R](y^{\imath}z^{\jmath})\right|\leq c_4\gamma^{(m+1)(2d)^m}s^{m}\mu \leq c\left(s_{-}^{m-1}\mu_{-}\right)^\frac{1}{L}.
\end{align*}

This completes the proof.
\end{proof}
\subsubsection{Estimate on $\Phi_+$}
Recall that $F$ is as in (\ref{eq3}) with the coefficients and its estimate is given by Lemma \ref{le2}. Then, we have the following estimate on $F$.
\begin{lemma}\label{le4}
There is a constant $c_4$ such that for all $|\ell|+|l|+|i|+|j|\leq m$,
\begin{align}\label{eF}
&\left|\partial_\xi^\ell\partial_x^l\partial_y^i\partial_z^jF\right|_{\hat D(s)\times G_+}\leq \left\{\begin{array}{l}
                             c_4s^{m-2|i|-|j|}\mu\Gamma(r-r_+),~~~~~~2|i|+|j|\leq m,\\
                             c_4\mu\Gamma(r-r_+),\quad\quad\quad\quad\quad m\leq2|i|+|j|.
                           \end{array}\right.
\end{align}
\end{lemma}
\begin{proof}
Let $$a(i,j)=\left\{\begin{array}{l}
                      s^{m-2|i|-|j|},\quad if~~2|i|+|j|\leq m, \\
                      0,\quad\quad\quad~~~  otherwise.
                    \end{array}\right.$$
By (\ref{eq3}) and Lemma \ref{le2}, we have
\begin{align*}
\left|\partial_\xi^\ell\partial_x^l\partial_y^i\partial_z^jF\right|_{\hat D(s)\times G_+}&\leq \sum_{0<|k|\leq K_+,2|\imath|+|\jmath|\leq m}|k|^l\left|\partial_\xi^\ell\partial_y^i\partial_z^j(F_{k\imath\jmath}y^\imath z^\jmath)\right|e^{|k|(r_++\frac{7}{8}(r-r_+))}\\
&\leq \sum_{0<|k|\leq K_+} c|k|^{|l|+|\ell|+|i|+|j|+(|\ell|+|i|+|j|+1)(2d)^\jmath\tau}s^{a(i,j)}\mu e^{-|k|\frac{r-r_+}{8}}\\
&\leq cs^{a(i,j)}\mu\Gamma(r-r_+).
\end{align*}
This proves (\ref{eF}).

This completes the proof.
\end{proof}
{To obtain the symplectic transformation stated in Theorem \ref{th1}, we need to extend the function $F$ smoothly to the domain $\hat D(\beta_0)\times G_0$.
\begin{lemma}\label{ext}
Assume \textsc{(H2)}-\textsc{(H3)}. Then $F$ and $\zeta_+-\zeta$ can be smoothly extended to functions of H\"{o}lder class $C^{m-1+\varrho,m-1+\varrho}(\hat D(\beta_0)\times G_0)$ and $C^{m-1+\varrho}$ respectively, where $0<\varrho<1$ is fixed. Moreover, there is a constant $c$ such that
\begin{align}\label{Fext}
\|F\|_{C^{m-1+\varrho,m-1+\varrho}(\hat D(\beta_0)\times G_0)}\leq c\mu\Gamma(r-r_+),\\\label{zetaext}
\|\zeta_+-\zeta\|_{C^{m-1+\varrho}(G_0)}\leq c(s_{-}^{m-1}\mu_{-})^\frac{1}{L},
\end{align}
where $\|F\|_{C^{m-1+\varrho,m-1+\varrho}(\hat D(\beta_0)\times G_0)}=:\left|\partial_\xi^\ell\partial_x^l\partial_y^i\partial_z^j F\right|_{\hat D(\beta_0)\times G_0}$, $|\ell|\leq m-1+\varrho$, $|l|+|i|+|j|\leq m-1+\varrho$.
\end{lemma}
\begin{proof}
It follows from the standard Whitney  extension theorem that $F$ and $\zeta_+-\zeta$ can be extended smoothly to functions $\tilde F$ and $\tilde{(\zeta_+-\zeta)}$ of H\"{o}lder class $C^{m-1+\varrho,m-1+\varrho}(\hat D(\beta_0\times G_0))$, $C^{m-1+\varrho}(G_0)$, respectively, such that
\begin{align*}
\|\tilde F\|_C^{m-1+\varrho,m-1+\varrho}(\hat D(\beta_0\times G_0))\leq c\|F\|_{C^{m,m}(\hat D(s)\times G_+)},\\
\|\tilde{\zeta_+-\zeta}\|_{C^{m-1+\varrho}(G_0)}\leq c\|\zeta_+-\zeta\|_{C^m(G_+)},
\end{align*}
where $c$ is a constant depending only on $m$, $\varrho$ and the dimensions $n, d$. Then this together with (\ref{eF}) in Lemma \ref{le4} and (\ref{eq23}) in Lemma \ref{le5} yield (\ref{Fext}) and (\ref{zetaext}).

This completes the proof.
\end{proof}
}
\begin{lemma}\label{le7}
In addition to \textsc{(H2)}-\textsc{(H3)}, assume that
\begin{align*}
&{\textsc{(H4)}}:c_3\left(s_{-}^{m-1}\mu_{-}\right)^{\frac{1}{L}}<\frac{1}{8}\alpha s,\\
&{\textsc{(H5)}}:c_4\mu\Gamma(r-r_+)<\frac{1}{4}(r-r_+),\\
&{\textsc{(H6)}}:c_4s^{m-1}\mu\Gamma(r-r_+)<\frac{1}{8}\alpha s,\\
&{\textsc{(H7)}}:c_3\mu\Gamma(r-r_+)+c_3\left(s^{m-1}\mu\right)^{\frac{1}{L}}<\beta-\beta_+.\\
\end{align*}
Then the following conclusions hold:
\begin{itemize}
\item[(1)]For all $0\leq t\leq 1$,
\begin{align}\label{eq26}
\phi_F^t&:D_{\frac{1}{4}\alpha}\rightarrow D_{\frac{1}{2}\alpha},\\\label{eq27}
\phi&:D_{\frac{1}{8}\alpha}\rightarrow D_{\frac{1}{4}\alpha},
\end{align}
are well defined.
\item[(2)] Let $\Phi_+=\phi_F^1\circ\phi$. Then for all $\xi\in G_+$,
    \begin{align}
    \Phi_+:
    \begin{array}{l}
    D_+\rightarrow D,\\
    \tilde D_+\rightarrow D(\beta,r).
    \end{array}
    \end{align}
\item[(3)]There is a constant $c_5$ such that
\begin{align*}
\left|\partial_\xi^\ell D_{(x,y,z)}^i(\phi_F^t-id)\right|_{{D}_{\frac{\alpha}{4}}\times G_+}&\leq c_5\mu\Gamma(r-r_+),~~~|\ell|+|i|\leq m,
\end{align*}
where $D_{(x,y,z)}^iF$ stands for the $|i|$-order partial derivative with respect to $(x,y,z)$.
\item[(4)]
\begin{align*}
\left|\partial_\xi^\ell D_{(x,y,z)}^i(\Phi_+-id)\right|_{\tilde{D}\times G_+}&\leq c_5\mu\Gamma(r-r_+),~~~|\ell|+|i|\leq m.
\end{align*}
\end{itemize}
\end{lemma}
\begin{proof}
(1)~~For $\forall (x,y,z)\in D_{\frac{1}{8}\alpha}$, we note by $\zeta_+\in B_{ (s_{-}^{m-1}\mu_{-})^{\frac{1}{L}}}(\zeta)$ and {(H4)} that
$$|z+\zeta_+-\zeta|<|z|+|\zeta_+-\zeta|<\frac{1}{8}\alpha s+c\left(s_{-}^{m-1}\mu_{-}\right)^{\frac{1}{L}}<\frac{1}{4}\alpha s,$$
which implies (\ref{eq27}).

To verify (\ref{eq26}), we denote $\phi_{F_1}^t$, $\phi_{F_2}^t$, $\phi_{F_3}^t$ as the components of $\phi_{F}^t$ in the $x$, $y$, $z$ coordinates, respectively. Let $X_F=\left(F_y,-F_x, \tilde J F_z\right)^\top$ be the vector field generated by $F$. Then
\begin{align}\label{eq28}
\phi_F^t=id+\int_0^tX_F\circ \phi_F^\lambda d\lambda,~~~~0\leq t\leq 1.
\end{align}
For any $(x,y,z)\in D_{\frac{1}{4}\alpha}$, let $t_*=\sup\left\{t\in[0,1]:\phi_F^t(x,y,z)\in D_\alpha\right\}$. Then for any $0\leq t\leq t_*$, in view of $(x,y,z)\in D_{\frac{1}{4}\alpha}$, (\ref{eF}) in Lemma \ref{le4}, {(H5)} and {(H6)},
\begin{align*}
\left|\phi_{F_1}^t(x,y,z)\right|_{D_{\frac{1}{4}\alpha}}&\leq|x|+\int_0^t\left|F_y\circ\phi_F^\lambda\right|_{D_{\frac{1}{4}\alpha}}d\lambda\\
&\leq r_++\frac{1}{8}(r-r_+)+c_4s^{m-2}\mu\Gamma(r-r_+)\\
&<r_++\frac{3}{8}(r-r_+),\\
\left|\phi_{F_2}^t(x,y,z)\right|_{D_{\frac{1}{4}\alpha}}&\leq|y|+\int_0^t\left|-F_x\circ\phi_F^\lambda\right|_{D_{\frac{1}{4}\alpha}}d\lambda\\
&\leq \frac{1}{4}\alpha s+c_4s^{m}\mu\Gamma(r-r_+)\\
&<\frac{3}{8}\alpha s,\\
\left|\phi_{F_3}^t(x,y,z)\right|_{D_{\frac{1}{4}\alpha}}&\leq|z|+\int_0^t\left|\tilde JF_z\circ\phi_F^\lambda\right|_{D_{\frac{1}{4}\alpha}}d\lambda\\
&\leq \frac{1}{4}\alpha s+c_4s^{m-1}\mu\Gamma(r-r_+)\\
&<\frac{3}{8}\alpha s.
\end{align*}
Thus, $\phi_F^t\in D_{\frac{1}{2}\alpha}\subset D_\alpha$, i.e. $t_*=1$ and (1) holds.

(2)~~Using \textsc{(H5)}, \textsc{(H7)}, (\ref{Fext}) and a similar argument as above, one sees that $\phi_F^t:\bar D_+\rightarrow D(\beta,r)$ is well defined for all $0\leq t\leq 1$, where $\bar D_+=D\left(r_++\frac{5}{8}(r-r_+),\beta_++c(s_{-}^{m-1}\mu_{-})^{\frac{1}{L}}\right)$. Thus, $\Phi_+:\tilde D_+\rightarrow D(\beta,r)$ is also well defined.

(3)~~{Note that
\begin{align}\label{xiXF}
\left|\partial_\xi^\ell X_F\right|_{\tilde D\times G_+}\leq c\left|\partial_\xi^\ell D_{(x,y,x)}F\right|_{\tilde D\times G_+}.
\end{align}
Using (\ref{eF}) in Lemma \ref{le4} and (\ref{eq28}), we immediately have
\begin{align*}
\left|\phi_F^t-id\right|_{{D}_{\frac{\alpha}{4}}\times G_+}\leq c\mu\Gamma(r-r_+).
\end{align*}
Differentiating (\ref{eq28}) yields
\begin{align*}
D_{(x,y,z)}\phi_F^t&=I_{2n+2d}+\int_0^t\left(D_{(x,y,z)}X_F\right)D_{(x,y,z)}\phi_F^\lambda d\lambda\\
&=I_{2n+2d}+\int_0^tJ\left(D_{(x,y,z)}^2F\right)D_{(x,y,z)}\phi_F^\lambda d\lambda,~~~~0\leq t\leq 1.
\end{align*}
It follows from Lemma \ref{le4}, (\ref{eq28}) and Gronwall Inequality that
\begin{align*}
\left|D_{(x,y,z)}\phi_F^t-I_{2n+2d}\right|_{{D}_{\frac{\alpha}{4}}\times G_+}&\leq\left|\int_0^tD_{(x,y,z)}X_F\circ \phi_F^\lambda D_{(x,y,z)}\phi_F^\lambda d\lambda\right|_{D_{\frac{\alpha}{4}}\times G_+}\\
&\leq\int_0^t\left|D_{(x,y,z)}X_F\circ \phi_F^\lambda\right|\left|D_{(x,y,z)}\phi_F^\lambda-I_{2n+2d}\right|_{\hat D(s)\times G_+}d\lambda\\
&\quad+\int_0^t\left|D_{(x,y,z)}X_F\circ \phi_F^\lambda\right|_{\hat D(s)\times G_+}d\lambda\\
&\leq c\mu\Gamma(r-r_+).
\end{align*}
Similarly, by using (\ref{eF}) in Lemma \ref{le4}, Gronwall's Inequality and (\ref{xiXF}) inductively, we have
\begin{align*}
\left|\partial_\xi^\ell D_{(x,y,z)}^i\phi_F^t\right|_{D_{\frac{\alpha}{4}}\times G_+}\leq c\mu\Gamma(r-r_+),~~~|\ell|+|i|\leq m.
\end{align*}

(4) Observe that
\begin{align*}
\Phi_+-id=(\phi_F^1-id)\circ\phi+\left(\begin{array}{c}
0\\
0\\
\zeta_+-\zeta\end{array}\right).
\end{align*}

Then,
\begin{align*}
\left|\partial_\xi^\ell D_{(x,y,z)}^i(\Phi_+-id)\right|_{\tilde{D}\times G_+}&\leq\left|\partial_\xi^\ell D_{(x,y,z)}^i(\phi_F^1-id)\right|_{\tilde{D}\times G_+}\left|\partial_\xi^\ell D_{(x,y,z)}^i\phi\right|_{\tilde{D}\times G_+}\\
&\quad+\left|\partial_\xi^\ell(\zeta_+-\zeta)\right|_{\tilde{D}\times G_+}\\
&\leq c\mu\Gamma(r-r_+)\left(s_{-}^{m-1}\mu_{-}\right)^{\frac{1}{L}}+\left(s_{-}^{m-1}\mu_{-}\right)^{\frac{1}{L}}\\
&\leq c\mu\Gamma(r-r_+)\alpha s+\alpha s\\
&\leq c\mu\Gamma(r-r_+),
\end{align*}
where the second equality follows from (\ref{Fext}) and (\ref{zetaext}) in Lemma \ref{ext}, the third equality follows from  \textsc{(H4)}, and the last equality follows from the definition of $\mu$ and $s$.}


This completes the proof.
\end{proof}

\subsubsection{Frequency property}\label{subsubF}
To estimate the new frequency, we first assume that
$${\textsc{(H8)}}: 3s K_+^{2\tau+1}\leq\left\{\frac{\gamma-\gamma_+}{\gamma_0},\frac{\gamma^2-\gamma_+^2}{\gamma_0^2}\right\}.$$
It is obvious that for all $0<|k|\leq K_+$, $\xi\in G_+$,
\begin{align*}
\left|\left\langle k,\nabla_y[R](0,\zeta_+-\zeta)\right\rangle\right|&\leq cK_+\gamma^{(m+1)(2d)^m}s^{m-2}\mu,\\
\left|\left\langle k,\nabla_y\bar g((\zeta_+-\zeta)^\jmath)\right\rangle\right|&\leq cK_+\gamma^{(m+1)(2d)^m}\left(s_{-}^{m-1}\mu_{-}\right)^{\frac{1}{L}}\leq cK_+\gamma^{(m+1)(2d)^m}\alpha s.
\end{align*}
Then this together with (\ref{omega+}) and {(H8)} yield
\begin{align*}
|\langle k,\omega_+\rangle|&\geq|\langle k,\omega\rangle|-\left|\left\langle k,\nabla_y[R](0,\zeta_+-\zeta)+\nabla_y\bar g(0,\zeta_+-\zeta)\right\rangle\right|\\
&\geq\frac{\gamma}{|k|^\tau}-\frac{\gamma-\gamma_+}{|k|^\tau}=\frac{\gamma_+}{|k|^\tau}.
\end{align*}
%
{Recall that
$$A_{\imath\jmath}^+=\sqrt{-1}\left\langle \frac{k}{|k|},\omega_+\right\rangle I_{(2d)^{\imath+\jmath}}+\frac{\tilde{S}_{\imath\jmath}^+}{|k|}
=A_{\imath\jmath}+\sqrt{-1}\left\langle \frac{k}{|k|},\omega_+-\omega\right\rangle I_{(2d)^{\imath+\jmath}}+\frac{\tilde{S}_{\imath\jmath}^+-\tilde{S}_{\imath\jmath}}{|k|}=:A_{\imath\jmath}+\tilde{A}_{\imath\jmath},$$
where
\begin{align*}
A_{\imath\jmath}&=\sqrt{-1}\left\langle \frac{k}{|k|},\omega\right\rangle I_{(2d)^{\imath+\jmath}}+\frac{\tilde{S}_{\imath\jmath}}{|k|},\\
\tilde{A}_{\imath\jmath}&=\sqrt{-1}\left\langle \frac{k}{|k|},\omega_+-\omega\right\rangle I_{(2d)^{\imath+\jmath}}+\frac{\tilde{S}_{\imath\jmath}^+-\tilde{S}_{\imath\jmath}}{|k|}.
\end{align*}
Then,
\begin{align}\label{tAA}
\overline{A_{\imath\jmath}^{+\top}}A_{\imath\jmath}^+=\overline{A_{\imath\jmath}^{\top}}A_{\imath\jmath}+\overline{A_{\imath\jmath}^{\top}}\tilde{A}_{\imath\jmath}+\overline{\tilde{A}_{\imath\jmath}^{\top}}A_{\imath\jmath}+\overline{\tilde{A}_{\imath\jmath}^{\top}}\tilde{A}_{\imath\jmath}.
\end{align}
By (\ref{eq23}) in Lemma (\ref{le5})
, we have
\begin{align}\label{om+-om}
\left|\sqrt{-1}\left\langle \frac{k}{|k|},\omega_+-\omega\right\rangle\right|&\leq \left(s_{-}^{m-1}\mu_{-}\right)^{\frac{1}{L}}\leq c\alpha s.
\end{align}
Note that $\tilde{S}_{\imath\jmath}^+-\tilde{S}_{\imath\jmath}$ is concerned with $\partial_z^2g_+(0)-\partial_z^2g(0)$, then, by (\ref{g+}), (\ref{R}) in Lemma \ref{le1},
\begin{align}\label{tS-S}
\frac{|\tilde{S}_{\imath\jmath}^+-\tilde{S}_{\imath\jmath}|}{|k|}\leq c{\gamma^{(m+1)(2d)^{\imath+\jmath}}s^{m-2}\mu}.
\end{align}
Thus, by (\ref{tAA}), (\ref{om+-om}), (\ref{tS-S}) and {(H8)},
\begin{align*}
\overline{A_{\imath\jmath}^{+\top}}A_{\imath\jmath}^+
&\geq \frac{\gamma^2}{|k|^{2\tau}}I_{(2d)^{\imath+\jmath}}-\frac{\gamma^2-\gamma_+^2}{3|k|^{2\tau}}I_{(2d)^{\imath+\jmath}}-\frac{\gamma^2-\gamma_+^2}{3|k|^{2\tau}}I_{(2d)^{\imath+\jmath}}-\frac{\gamma^2-\gamma_+^2}{3|k|^{2\tau}}I_{(2d)^{\imath+\jmath}}\geq\frac{\gamma_+^2}{|k|^{2\tau}}I_{(2d)^{\imath+\jmath}}.
\end{align*}}

\subsubsection{Estimate on $P_+$}
In the following, we estimate the next step $P_+$.
\begin{lemma}\label{le8}
Denote
\begin{align*}
\Delta&=\alpha^{m+1}s^{m+1}\mu\left(s^{m-2}\mu\Gamma^2(r-r_+)+\Gamma(r-r_+)\right)+\gamma^{(m+1)(2d)^m}s^{2m-2}\mu^2\Gamma(r-r_+)\\
&\quad+\gamma^{(m+1)(2d)^m}s^{m+1}\mu c.
\end{align*}
Assume \textsc{(H0)}-\textsc{(H8)}. Then there is a constant $c_6$ such that for all $|\ell|\leq m$
\begin{equation}\label{eq29}
\left|\partial_\xi^\ell\bar P_+\right|_{D_+\times G_+}\leq c_6\Delta.
\end{equation}
Moreover, if
$${\textsc{(H9)}}:c_6\Delta\leq\gamma_+^{(m+1)(2d)^m}s_+^{m}\mu_+,$$
then
\begin{equation}\label{ep+}
\left|\partial_\xi^\ell P_+\right|_{D_+\times G_+}\leq \gamma_+^{(m+1)(2d)^m}s_+^{m}\mu_+.
\end{equation}
\end{lemma}
\begin{proof}
Recall the definition of $Q$ as in (\ref{Q}). 
Observe by (\ref{barg}) and (\ref{eF}) that
\begin{align}\label{eQ}
\left|\partial_\xi^\ell Q\right|_{D_{\frac{1}{4}\alpha}\times G_+}
\leq c(\alpha s)^{m+1}\mu\Gamma(r-r_+)
\end{align}
and
\begin{align}\label{eQF}
\left|\partial_\xi^\ell\{Q,F\}\right|_{D_{\frac{1}{4}\alpha}\times G_+}
&\leq c\alpha^{m+1}s^{m+1}s^{m-2}\mu^2\Gamma^2(r-r_+)+c\alpha^{m+1}s^{m-1}s^m\mu^2\Gamma^2(r-r_+)\notag\\
&~~~+c\alpha^{m+1} s^ms^{m-1}\mu^2\Gamma^2(r-r_+)\notag\\
&\leq
c\alpha^{m+1} s^{2m-1}\mu^2\Gamma^2(r-r_+). 
\end{align}
Using (\ref{R}) in Lemma \ref{le1} and (\ref{eF}) in Lemma \ref{le4}, we also have
\begin{align*}
\left|\partial_\xi^\ell[R]\right|_{\hat D\times G_+}+\left|\partial_\xi^\ell R\right|_{\hat D\times G_+}\leq\gamma^{(m+1)(2d)^m}s^{m}\mu
\end{align*}
and
\begin{align}\label{eRF}
\left|\partial_\xi^\ell\{[R],F\}\right|_{D_{\frac{1}{4}\alpha}\times G_+}+\left|\partial_\xi^\ell\{R,F\}\right|_{D_{\frac{1}{4}\alpha}\times G_+}\leq c\gamma^{(m+1)(2d)^m}s^{2m-2}\mu^2\Gamma(r-r_+).
\end{align}
Let us recall the definition of $\bar P_+$ in (\ref{eq16}), that is
$$ \bar P_+=\int_0^1\{R_t,F\}\circ\phi_F^tdt+(P-R)\circ\phi_F^1+Q,~~~
R_t=(1-t)Q+(1-t)[R]+tR.$$
Then by (\ref{P-R}), (\ref{eQ}), (\ref{eQF}) and (\ref{eRF}),
\begin{align*}
\left|\partial_\xi^\ell\bar P_+\right|_{D_{\frac{1}{4}\alpha}\times G_+}&\leq c\alpha^{m+1}s^{2m-1}\mu^2\Gamma^2(r-r_+)+c\gamma^{(m+1)(2d)^m}s^{2m-2}\mu^2\Gamma(r-r_+)\\
&\quad+c\gamma^{(m+1)(2d)^m}s^{m+1}\mu+c(\alpha s)^{m+1}\mu\Gamma(r-r_+)\\
&\leq\alpha^{m+1}s^{m+1}\mu(s^{m-2}\mu\Gamma^2(r-r_+)+\Gamma(r-r_+))\\
&\quad+\gamma^{(m+1)(2d)^m}s^{m+1}\mu(s^{m-3}\mu\Gamma(r-r_+)+c),
\end{align*}
which implies (\ref{eq29}).

Moreover, by {(H9)} and the definition of $P_+$ as in (\ref{eq18}), we see
\begin{align*}
\left|\partial_\xi^\ell P_+\right|_{D_+\times G_+}
&\leq\gamma_+^{(m+1)(2d)^m}s_+^m\mu_+.
\end{align*}

This completes the proof.
\end{proof}
This completes one cycle of KAM steps.
\section{Proof of Theorem \ref{th1}}\label{sec:4}
\subsection{Iteration lemma}
In this section, we will prove an iteration lemma which guarantees the inductive construction of the transformations in all KAM steps.

Let $r_0,\beta_0,\gamma_0,s_0,\alpha_0,
\mu_0,H_0,N_0,P_0$ be given at the beginning of Section \ref{sec:3} and let 
$D_0=D(s_0,r_0)$, $\tilde D_0=D(\beta_0,r_0)$, $K_0=0$, $\Phi_0=id$. We define the following sequence inductively for all $\nu=1,2,\cdots$:
\begin{align*}
r_\nu&=r_0\left(\frac{1}{2}+\frac{1}{2^{\nu+1}}\right),\\
\beta_\nu&=\beta_0\left(\frac{1}{2}+\frac{1}{2^{\nu+1}}\right),\\
\gamma_\nu&=\gamma_0\left(\frac{1}{2}+\frac{1}{2^{\nu+1}}\right),\\
s_\nu&=\frac{1}{8}\alpha_{\nu-1}s_{\nu-1},\\
\alpha_\nu&=s_\nu^{2\rho}=s_\nu^{\frac{1}{m+1}},\\
\mu_\nu&=8^mc_0\mu_{\nu-1}s_{\nu-1}^{\rho},\\
K_\nu&=\left(\left[\log\left(\frac{1}{s_{\nu-1}}\right)\right]+1\right)^{3\eta},\\
D_\nu&=D(s_\nu,r_\nu),\\
\tilde{D}_\nu&=D\left(\beta_\nu, r_\nu+\frac{5}{8}(r_{\nu-1}-r_\nu)\right),\\
D(s_\nu)&=\{(y,z):|y|<s_\nu^2,|z|<s_\nu\},\\
G_{\nu}&=\Bigg\{\xi\in G_{\nu-1}: |\langle k,\omega_{\nu-1}\rangle|>\frac{\gamma}{|k|^\tau},\overline{A_{\imath\jmath}^{{\nu-1}\top}} A_{\imath\jmath}^{\nu-1}>\frac{\gamma^2}{|k|^{2\tau}}I_{(2d)^{\imath+\jmath}},~ for~ k\in\mathbb{Z}^n, 0<|k|<K_\nu,\\
&\quad~~2|\imath|+|\jmath|\leq m, 1\leq|\jmath|\Bigg\},~where ~A_{\imath\jmath}^{\nu-1}~ is~ defined ~ as ~in~ (\ref{A}).
\end{align*}
\begin{lemma}\label{le9}
Denote
$$\mu_*=s^{2}\varepsilon^{\frac{1}{4(m+1)}}\gamma_0^{2(m+1)(2d)^m}.$$
If $\varepsilon$ is small enough, then the KAM step described in Section \ref{sec:3} is valid for all $\nu=0,1,\cdots$, resulting in the sequences
$$ e_\nu, \omega_\nu, \tilde h_\nu, g_\nu, \bar{g}_\nu, H_\nu, P_\nu, \Phi_\nu,$$
$\nu=1,2,\cdots,$ with the following properties:
\begin{itemize}
\item[(1)] For all $|\ell|\leq m$,
\begin{align}\label{e+-e}
\left|\partial_\xi^\ell({e}_{\nu+1}-{e}_{\nu})\right|_{D(s_{\nu+1})\times G_{\nu+1}}&\leq\frac{\mu_*^{\frac{1}{2}}}{2^{\nu}},\\\label{e+-e0}
\left|\partial_\xi^\ell({e}_\nu-{e}_{0})\right|_{D(s_{\nu})\times G_\nu}&\leq2\mu_*^{\frac{1}{2}},\\\label{iomega+-omega}
\left|\partial_\xi^\ell({\omega}_{\nu+1}-{\omega}_{\nu})\right|_{D(s_{\nu+1})\times G_{\nu+1}}&\leq\frac{\mu_*^{\frac{1}{2}}}{2^{\nu}},\\\label{omega+-omega0}
\left|\partial_\xi^\ell({\omega}_\nu-{\omega}_{0})\right|_{D(s_{\nu})\times G_\nu}&\leq2\mu_*^{\frac{1}{2}},\\\label{fth+-fth}
\left|\partial_\xi^\ell(\tilde{h}_{\nu+1}-\tilde{h}_{\nu})\right|_{D(s_{\nu+1})\times G_{\nu+1}}&\leq\frac{\mu_*^{\frac{1}{2}}}{2^{\nu}},\\\label{fth+-fth0}
\left|\partial_\xi^\ell(\tilde{h}_\nu-\tilde{h}_{0})\right|_{D(s_{\nu})\times G_\nu}&\leq2\mu_*^{\frac{1}{2}},\\\label{fg+-fg}
\left|\partial_\xi^\ell(g_{\nu+1}-g_{\nu})\right|_{D(s_{\nu+1})\times G_{\nu+1}}&\leq\frac{\mu_*^{\frac{1}{2}}}{2^{\nu}},\\\label{fg+-fg0}
\left|\partial_\xi^\ell(g_\nu-g_{0})\right|_{D(s_{\nu})\times G_\nu}&\leq2\mu_*^{\frac{1}{2}},\\\label{fbg+-fbg}
\left|\partial_\xi^\ell(\bar{g}_{\nu+1}-\bar{g}_{\nu})\right|_{D(s_{\nu+1})\times G_{\nu+1}}&\leq\frac{\mu_*^{\frac{1}{2}}}{2^{\nu}},\\\label{fbg+-fbg0}
\left|\partial_\xi^\ell(\bar{g}_\nu-\bar{g}_{0})\right|_{D(s_{\nu})\times G_\nu}&\leq2\mu_*^{\frac{1}{2}},\\\label{z+-z}
\left|\partial_\xi^\ell(\zeta_{\nu+1}-\zeta_{\nu})\right|_{G_{\nu+1}}&\leq (s_{\nu-1}^{m-1}\mu_{\nu-1})^{\frac{1}{L}},\\\label{p+}
\left|\partial_\xi^\ell P_\nu\right|_{D_{\nu}\times G_\nu}&\leq\gamma_\nu^{(m+1)(2d)^m}s_\nu^m\mu_\nu.
\end{align}

\item[(2)] $G_{\nu+1}=\left\{\xi\in G_{\nu}: |\langle k,\omega_{\nu}\rangle|>\frac{\gamma_\nu}{|k|^\tau},~\overline{A_{\imath\jmath}^{{\nu}\top}} A_{\imath\jmath}^{\nu}>\frac{\gamma_\nu^2}{|k|^{2\tau}}I_{(2d)^{\imath+\jmath}},for~ k\in\mathbb{Z}^n,~ K_{\nu}<|k|\leq K_{\nu+1},~ 2|\imath|+|\jmath|\leq m,~ 1\leq|\jmath|\right\}$.
\item[(3)] {There exists a family of symplectic coordinate transformations $\Phi_{\nu+1}:\tilde{D}_{\nu+1}\times G_{\nu+1}\rightarrow \tilde{D}_{\nu}$ and a subset $G_{\nu+1}\subset G_\nu$,
    \begin{align}\label{Gnu+1} G_{\nu+1}=G_\nu\setminus\bigcup_{K_{\nu}<|k|\leq K_{\nu+1}}\mathcal{R}_{k}^{\nu+1}(\gamma_{\nu}),
    \end{align}
    where
\begin{align*}
    \mathcal{R}_{k}^{\nu+1}(\gamma_{\nu})=&\left\{\xi\in G_{\nu}: |\langle k,\omega_{\nu}\rangle|\leq\frac{\gamma_\nu}{|k|^\tau},~\overline{A_{\imath\jmath}^{{\nu}\top}} A_{\imath\jmath}^{\nu}\leq\frac{\gamma_\nu^2}{|k|^{2\tau}}I_{(2d)^{\imath+\jmath}}, for~ k\in\mathbb{Z}^n,~ K_{\nu}<|k|\leq K_{\nu+1},~ 2|\imath|+|\jmath|\leq m,~ 1\leq|\jmath|\right\},
    \end{align*}
    such that on $\tilde{D}_{\nu+1}\times G_{\nu+1}$ \begin{equation*}
H_{\nu+1}=H_\nu\circ\Phi_{\nu+1}=N_{\nu+1}+P_{\nu+1}
\end{equation*}
and
\begin{align}\label{eq36}
\left|\partial_\xi^\ell\partial_{(x,y,z)}^i(\Phi_{\nu+1}-id)\right|_{\tilde{D}_{\nu+1}\times G_{\nu+1}}\leq\frac{\mu_*^{\frac{1}{2}}}{2^{\nu}}.
\end{align}}

\end{itemize}
\end{lemma}
\begin{proof}
The proof amounts to the verification of {(H0)}-{(H9)} for all $\nu$.
For simplicity, we let $r_0=1$. We can make $s_0$ small by picking $\varepsilon_0$ sufficiently small.
So, we see that {(H0)} holds and {(H1)}-{(H9)} hold for $\nu=0$.

Note that
\begin{align}\label{munu}
\mu_\nu&=(8^mc_0)^\nu \left(\frac{1}{8}\right)^{\frac{(m+2)\left(\left(1+\frac{1}{m+1}\right)^\nu-1\right)-(\nu-1)}{2}} s_0^{\frac{1}{2}\left(\left(1+\frac{1}{m+1}\right)^\nu-1\right)}\mu_0,\\\label{snu}
s_\nu&=(\frac{1}{8})^{(m+1)\left(\left(1+\frac{1}{m+1}\right)^\nu-1\right)} s_0^{\left(1+\frac{1}{m+1}\right)^\nu}.
\end{align}
Let $\theta\gg1$ be fixed and $s_0$ be small enough so that
\begin{equation}\label{eq38}
s_0<\left(\frac{8}{\theta}\right)^{\frac{1}{2\rho}}<1.
\end{equation}
Then
\begin{align}
s_1&=\frac{1}{8}s_0^{1+2\rho}<\frac{1}{\theta}s_0<1,\notag\\
s_2&=\frac{1}{8}s_1^{1+2\rho}<\frac{1}{\theta}s_1<\frac{1}{\theta^2}s_0,\notag\\
\vdots\notag\\\label{eq39}
s_\nu&=\frac{1}{8}s_{\nu-1}^{1+2\rho}<\cdots<\frac{1}{\theta^\nu}s_0.
\end{align}
Denote
\begin{equation*}
\Gamma_\nu=\Gamma(r_\nu-r_{\nu+1}).
\end{equation*}
We notice that
\begin{equation}\label{eq40}
\frac{r_\nu-r_{\nu+1}}{r_0}=\frac{1}{2^{\nu+2}}.
\end{equation}
Since
\begin{align*}
\Gamma_\nu&\leq\int_1^\infty t^{(m+1)(2d)^m\tau+m}e^{-\frac{t}{2^{\nu+5}}}dt\\
&\leq((m+1)(2d)^m\tau+m)!2^{(\nu+5)((m+1)(2d)^m\tau+m)},
\end{align*}
it is obvious that for $\theta$ large enough,
\begin{align}\label{srho}
s_\nu^{\frac{\rho}{2}}\Gamma_\nu^i<s_\nu^{\frac{\rho}{2}}(\Gamma_\nu^i+\Gamma_\nu)\leq1,~~~i=1,2,
\end{align}
and
\begin{equation*}
s_\nu\Gamma_\nu\leq s_\nu^{1-\rho}\leq\frac{s_0^{1-\rho}}{\theta^{(1-\rho)\nu}}.
\end{equation*}

Recall $\alpha_\nu=s_\nu^{\frac{1}{m+1}}$, then $\alpha_\nu^{m+1}=s_\nu$. In view of $s_0=s\varepsilon^{\frac{1}{8(m+1)}}\gamma_0^{(m+1)(2d)^m}$, we have $\gamma_0^{(m+1)(2d)^m}>s_0$. 
It is obvious by the definition of $\gamma_\nu$ and (\ref{snu}) that $\gamma_\nu^{(m+1)(2d)^m}>s_\nu$ if $\varepsilon_0$ is small enough. Moreover,
\begin{align*}
\gamma_\nu^{(m+1)(2d)^m}s_\nu^{\frac{1}{4(m+1)}}&=\left(\frac{1}{2}+\frac{1}{2^{\nu+1}}\right)^{(m+1)(2d)^m}\left(\frac{1}{8}\right)^{\frac{1}{4}\left(\left(1+\frac{1}{m+1}\right)^\nu-1\right)}s_0^{\frac{1}{4(m+1)}\left(1+\frac{1}{m+1}\right)^\nu}\\
&\leq\left(\frac{1}{2}+\frac{1}{2^{\nu+2}}\right)^{(m+1)(2d)^m}=\gamma_{\nu+1}^{(m+1)(2d)^m}
\end{align*}
as $s_0$ small enough.
This together with the definition of $\Delta_\nu$ as in Lemma \ref{le8} yields
\begin{align*}
\Delta_\nu&=\alpha_\nu^{m+1}s_\nu^{m+1}\mu_\nu\left(s_\nu^{m-2}\mu_\nu\Gamma^2(r_\nu-r_{\nu+1})+\Gamma(r_\nu-r_{\nu+1})\right)\\
&\quad+\gamma_\nu^{(m+1)(2d)^m}s_\nu^{2m-2}\mu_\nu^2\Gamma(r_\nu-r_{\nu+1})+\gamma_\nu^{(m+1)(2d)^m}s_\nu^{m+1}\mu_\nu c\\
&\leq\gamma_\nu^{(m+1)(2d)^m}s_\nu^{\frac{1}{4(m+1)}}s_\nu^{(1+\frac{1}{m+1})m}s_\nu^{\frac{1}{2(m+1)}}\mu_\nu (s_\nu^{\frac{1}{4(m+1)}}s_\nu^{m-2}\mu_\nu\Gamma^2(r_\nu-r_{\nu+1})\\
&\quad+s_\nu^{\frac{1}{4(m+1)}}\Gamma(r_\nu-r_{\nu+1})+s_\nu^{\frac{1}{4(m+1)}}s_\nu^{m-3}\mu\Gamma(r_\nu-r_{\nu+1})+s_\nu^{\frac{1}{4(m+1)}}c)\\
&\leq\gamma_{\nu+1}^{(m+1)(2d)^m}s_{\nu+1}^m\mu_{\nu+1},
\end{align*}
which implies that {(H9)} holds for all $\nu\geq1$.

We note that for any constant $a>0$, $b>1$, $s^a\left(\log\frac{1}{s}+1\right)^b\rightarrow0$ as $s\rightarrow0$, then
\begin{align*}
3\gamma_0^2s_\nu K_{\nu+1}^{2\tau+1}&\leq 3\gamma_0^2 s_\nu\left(\left[\log\frac{1}{s_\nu}\right]+1\right)^{3\eta(2\tau+1)}\\
&\leq\left(\frac{1}{2}+\frac{1}{2^{\nu+1}}\right)^2-\left(\frac{1}{2}+\frac{1}{2^{\nu+2}}\right)^2\\
&=\gamma_\nu^2-\gamma_{\nu+1}^2,
\end{align*}
if $\varepsilon_0$ is small enough. Hence, {\textsc{(H8)}} holds.

By (\ref{eq40}) and (\ref{srho}), it is easy to verify that {(H5)}-{(H7)} hold for all $\nu\geq1$ if $\theta$ is large enough and $\varepsilon_0$ is small enough.


It is obvious by $m\geq\frac{L+\sqrt{L^2+16L+16}}{4},~L\geq2$ as in (\ref{m}) that
$$\frac{m-1}{L}>2\frac{m+2}{m+1}>\left(\frac{m+2}{m+1}\right)^2,$$
then $$s_{-}^\frac{m-1}{L}<s_{-}^{\left(\frac{m+2}{m+1}\right)^2},$$
i.e.,
$$\left(s_{-}^{m-1}\mu_{-}\right)^\frac{1}{L}<\left(\frac{1}{8}\right)^{1+\frac{m+2}{m+1}}s_{-}^{\left(\frac{m+2}{m+1}\right)^2}=\frac{1}{8}s^\frac{m+2}{m+1}=\frac{1}{8}\alpha s,$$
which implies {(H4)}.

To verify {\textsc{(H3)}}, we observe by (\ref{snu}) and (\ref{eq39}) that
$$4(M^*+2)s_\nu K_{\nu+1}^{\tau+1}<\frac{s_0}{2^{\nu+2}}<\frac{\gamma_0}{2^{\nu+2}},$$
as $\theta$ is large enough, which verifies {\textsc{(H3)}} for all $\nu\geq1$.

Let $\theta^{1-\rho}\geq2$ in (\ref{eq38}), (\ref{eq39}). We have that for all $\nu\geq1$
\begin{align}\label{eq44}
s_\nu&\leq\frac{s_0}{2^\nu}\leq\frac{\mu_*^\frac{1}{2}}{2^\nu}.
\end{align}
The verification of {\textsc{(H2)}} follows from (\ref{eq44}) and an inductive application of (\ref{eq25}) in Lemma \ref{le5} for all $\nu=0,1,\cdots.$



Since $(1+\rho)^\eta>2$, we have
\begin{align*}
\frac{r_0}{2^{\nu+6}}\left(\left[\log\frac{1}{s}\right]+1\right)^\eta&\geq\frac{r_0}{2^{\nu+6}}\left(-(1+\rho)^\nu\log s_0\right)^\eta\\
&\geq-\frac{r_0}{2^{\nu+6}}(1+\rho)^{\eta\nu}(\log\mu_0)^\eta\\
&\geq1.
\end{align*}
It follows from above that
\begin{align*}
&\log(n+1)!+(\nu+6)n\log2+3n\eta\log\left(\left[\log\frac{1}{s}\right]+1\right)-\frac{r_0}{2^{(\nu+6)}}\left(\left[\log\frac{1}{s}\right]+1\right)^{3\eta}\\
&\leq\log(n+1)!+(\nu+6)n\log2+3n\eta\log\left(\log\frac{1}{s}+2\right)-\left(\log\frac{1}{s}\right)^{2\eta}\\
&\leq-\log\frac{1}{s},
\end{align*}
as $\mu$ is small, which is ensured by making $\varepsilon$ small.
Thus,
\begin{equation*}
\int_{K_{\nu+1}}^\infty t^{n}e^{-\frac{tr_0}{2^{\nu+6}}}dt\leq(n+1)!2^{(\nu+6)n}K_{\nu+1}^{n}e^{-\frac{K_{\nu+1}}{2^{\nu+6}}}\leq s,
\end{equation*}
i.e. {(H1)} holds.

Above all, the KAM steps described in Section \ref{sec:3} are valid for all $\nu$, which give the desired sequences stated in the lemma.

Now, (\ref{e+-e}), (\ref{iomega+-omega}), (\ref{fth+-fth}), (\ref{fg+-fg}) and  (\ref{fbg+-fbg}) follow from (\ref{eq25}), (\ref{eq48}) and (\ref{bg+-bg}) in Lemma \ref{le5} and (\ref{eq44}); by adding up (\ref{e+-e}), (\ref{omega+-omega}), (\ref{fth+-fth}), (\ref{fg+-fg}) and  (\ref{fbg+-fbg}) for all $\nu=0,1,\cdots$, we can get (\ref{e+-e0}), (\ref{omega+-omega0}) , (\ref{fth+-fth0}), (\ref{fg+-fg0}) and  (\ref{fbg+-fbg0}) respectively; (\ref{p+}) follows from (\ref{ep+}) in Lemma \ref{le8}; (\ref{z+-z})  follows from Lemma \ref{le5}.

{Note that (2) automatically holds for $\nu=0$. We now let $\nu>0$. By subsubsection \ref{subsubF}, it is obvious that
$$G_{\nu}=\left\{\xi\in G_{\nu}: |\langle k,\omega_{\nu}\rangle|>\frac{\gamma_\nu}{|k|^\tau},~\overline{A_{\imath\jmath}^{{\nu}\top}} A_{\imath\jmath}^{\nu}>\frac{\gamma_\nu^2}{|k|^{2\tau}}I_{(2d)^{\imath+\jmath}},for~ k\in\mathbb{Z}^n,~0<|k|\leq K_{\nu},~ 2|\imath|+|\jmath|\leq m,~ 1\leq|\jmath|\right\}.$$
Denote
$$\hat G_{\nu+1}=\left\{\xi\in G_{\nu}: |\langle k,\omega_{\nu}\rangle|>\frac{\gamma_\nu}{|k|^\tau},~\overline{A_{\imath\jmath}^{{\nu}\top}} A_{\imath\jmath}^{\nu}>\frac{\gamma_\nu^2}{|k|^{2\tau}}I_{(2d)^{\imath+\jmath}},for~ k\in\mathbb{Z}^n,~K_\nu<|k|\leq K_{\nu+1},~ 2|\imath|+|\jmath|\leq m,~ 1\leq|\jmath|\right\}.$$
Then
\begin{align*}
G_{\nu+1}=&\left\{\xi\in G_{\nu}: |\langle k,\omega_{\nu}\rangle|>\frac{\gamma_\nu}{|k|^\tau},~\overline{A_{\imath\jmath}^{{\nu}\top}} A_{\imath\jmath}^{\nu}>\frac{\gamma_\nu^2}{|k|^{2\tau}}I_{(2d)^{\imath+\jmath}},for~ k\in\mathbb{Z}^n,~0<|k|\leq K_{\nu+1},~ 2|\imath|+|\jmath|\leq m,~ 1\leq|\jmath|\right\}\\
=&G_\nu\cap\hat G_{\nu+1}=\hat G_{\nu+1},
\end{align*}
which implies (\ref{Gnu+1}).}

This completes the proof.
\end{proof}

\subsection{Convergence}
The convergence is standard. For the sake of completeness, we briefly give the outline of the proof.
Let
\begin{align*}
\Psi^\nu=\Phi_0\circ\Phi_1\circ\cdots\circ\Phi_\nu.
\end{align*}
Recalling that $\Phi_{\nu+1}:\tilde{D}_{\nu+1}\rightarrow \tilde{D}_{\nu}$ in
Lemma \ref{le9}, we have
\begin{align*}
\Psi^\nu&:\tilde{D}_\nu\rightarrow \tilde{D}_0,\\
H_0\circ\Psi^\nu&=H_\nu=N_\nu+P_\nu,
\end{align*}
$\nu=0,1,\cdots,$ where $\Psi^0=id$. { Let $G_*=\bigcap_{\nu=0}^{\infty} G_\nu$. First, we show the uniform convergence of $\Psi^\nu$ on $D\left(\frac{\beta_0}{2},\frac{r_0}{2}\right)\times G_*$. Note that
\begin{align*}
\Psi^\nu=id+\sum_{i=1}^\nu\left(\Psi^i-\Psi^{i-1}\right),
\end{align*}
where, for each $i=1,2,\cdots$,
\begin{align*}
\Psi^i-\Psi^{i-1}&=\Phi_0\circ\cdots\circ\Phi_i-\Phi_0\circ\cdots\circ\Phi_{i-1}\\
&\leq\int_0^1D(\Phi_0\circ\cdots\circ\Phi_{i-1})(id+\theta(\Phi_i-id))d\theta(\Phi_i-id).
\end{align*}
Since, by (\ref{eq36}), on $D(\frac{\beta_0}{2},\frac{r_0}{2})\times G_*$,
\begin{align*}
&\left|D(\Phi_0\circ\cdots\circ\Phi_{i-1})(id+\theta(\Phi_i-id))\right|\\
&\leq|D\Phi_0(\Phi_1\circ\cdots\circ\Phi_{i-1})(id+\theta(\Phi_i-id))|\cdots|D\Phi_{i-1}(id+\theta(\Phi_i-id))|\\
&\leq \left(1+\mu_*^{\frac{1}{2}}\right)\left(1+\frac{\mu_*^{\frac{1}{2}}}{2}\right)\cdots\left(1+\frac{\mu_*^{\frac{1}{2}}}{2^{i-1}}\right)\leq e^{1+\frac{1}{2}+\cdots+\frac{1}{2^{i-1}}}\leq e^2,
\end{align*}
we have
\begin{align*}
|\Psi^i-\Psi^{i-1}|_{D(\frac{\beta_0}{2},\frac{r_0}{2})\times G_*}\leq e^2|\Phi_i-id|_{D(\frac{\beta_0}{2},\frac{r_0}{2})\times G_*}\leq e^2 \frac{\mu_*^{\frac{1}{2}}}{2^{i}},
\end{align*}
for all $i=1,2,\cdots.$ Thus, $\Psi^\nu$ converges uniformly on $D\left(\frac{\beta_0}{2},\frac{r_0}{2}\right)\times G_*$. We denote its limit by $\Psi^*$. } Then
\begin{align}
\Psi^*=\Psi^0+\sum_{i=1}^\infty\left(\Psi^i-\Psi^{i-1}\right),
\end{align}
which is uniformly close to the identity.
By a standard argument using the Whitney Extension Theorem, one can further show that $\Psi^*$ is Whitney smooth with respect to $\xi\in G_* $ (see \cite{han,li2,li} for details).
By Lemma \ref{le9},we see that $e_\nu$, $ \omega_\nu$, $\tilde h_\nu$, $ g_\nu$, $\bar{g}_\nu$ and $\zeta_\nu$ are uniformly convergent and denote the limits by  $e_*$, $ \omega_*$, $\tilde h_*$, $g_*$, $\bar{g}_*$ and $\zeta_*$, respectively. It follows from Lemma \ref{le3} that $$\nabla g_*(0)=\cdots=\nabla g_\nu(0)=\cdots=\nabla g_0(\zeta_0)=0.$$
Then, $N_\nu$ converge uniformly to
\begin{align*}
N_*=e_*+\langle\omega_*,y\rangle+\tilde h_* (y)+g_*(z)+\bar{g}_*(y, z),
\end{align*}
with
\begin{align}
\tilde h_*(y)&=O(|y|^2),\notag\\
g_*(z)&=g(z)+\sum_{i=0}^{\infty}\gamma_i^{(m+1)(2d)^m}s_i^{m-2}\mu_i O(|z|^2)\notag\\\label{g*2}
&=g(z)+\varepsilon^{\frac{(m-2)(5m+4)}{8m(m+1)}} O(|z|^2),\\
\bar g_*(y,z,\xi)&=\sum_{2|\imath|+|\jmath|\leq m,1\leq |\imath|, |\jmath| }\bar g_{\imath\jmath}(\xi)y^\imath z^\jmath,\notag
\end{align}
where (\ref{g*2}) follows from (\ref{eq44}) and the definition of $\mu_*$ in Lemma \ref{le9}.

Hence
\begin{align*}
P_\nu=H_0\circ\Psi^\nu-N_\nu
\end{align*}
converges uniformly to
\begin{equation*}
P_*=H_0\circ\Psi^*-N_*.
\end{equation*}
Since
\begin{align*}
\left|P_\nu\right|_{D_\nu}\leq\gamma_\nu^{(m+1)(2d)^m}s_\nu^m\mu_\nu,
\end{align*}
the Cauchy estimate implies that
$$\left|\partial_y^i\partial_z^jP_\nu\right|_{D\left(\frac{s_\nu}{2},r_\nu\right)}\leq\gamma_\nu^{(m+1)(2d)^m}s_\nu^{(m-2i-j)}\mu_\nu,$$
for all $2|i|+|j|\leq m$. By (\ref{munu}) and (\ref{snu}), it is easy to see that the right hand side of the above converges to $0$ as $\nu\rightarrow\infty$. Hence, on ${D\left(0,\frac{r_0}{2}\right)}\times G_*$,
$$\partial_y^j\partial_z^jP_*\big|_{(y,z)=(0,0)}=0,$$
for all $x\in \mathbb{T}^n$, $\omega\in G_*$, $2|i|+|j|\leq m$.
It follows that
for each $\xi\in G_*$, $\mathbb{T}^n\times\{0\}\times\{0\}$ is an analytic, quasi-periodic, invariant
$n$-torus associated to the Hamiltonian $H_\infty=H\circ\Psi_\infty$ with frequency $\omega_*$.

\subsection{Measure estimate}
{\begin{lemma}\label{le10}
$$\left|G_0\backslash G_*\right|\rightarrow0,~as~\varepsilon\rightarrow0.$$
\end{lemma}
\begin{proof}
Since all the eigenvalues of normal direction are $0$ in $0$-th step, the measure estimate of $G_1$ in the first step can be reduced to the estimate of $\langle {k},\omega\rangle$. Moreover, if $\tilde S_{\imath\jmath}^\nu\equiv0$, the measure estimate of $G_\nu$ in every step also can be reduced to the estimate of $\langle {k},\omega\rangle$.  For details, see \cite{li2}. Otherwise, if $\tilde S_{\imath\jmath}^\nu\neq0$, the measure estimate of $G_\nu$ is referred to \cite{qian3}. For the sake of completeness, we give the outline of the proof.

Recall \begin{align*}
    \mathcal{R}_{k}^{\nu+1}(\gamma_{\nu})=&\Bigg\{\xi\in G_{\nu}: |\langle k,\omega_{\nu}\rangle|\leq\frac{\gamma_\nu}{|k|^\tau},~\overline{A_{\imath\jmath}^{{\nu}\top}} A_{\imath\jmath}^{\nu}\leq\frac{\gamma_\nu^2}{|k|^{2\tau}}I_{(2d)^{\imath+\jmath}}, \\& for~ k\in\mathbb{Z}^n,~ K_{\nu}<|k|\leq K_{\nu+1},~ 2|\imath|+|\jmath|\leq m,~ 1\leq|\jmath|\Bigg\}\\
    \subset&\mathcal{S}_1\cup\mathcal{S}_2,
    \end{align*}
where \begin{align*}
\mathcal{S}_1&=\left\{\xi\in G_{\nu}: |\langle k,\omega_{\nu}\rangle|\leq\frac{\gamma_\nu}{|k|^\tau}\right\},\\
\mathcal{S}_2&=\left\{\xi\in G_{\nu}: \overline{A_{\imath\jmath}^{{\nu}\top}} A_{\imath\jmath}^{\nu}\leq\frac{\gamma_\nu^2}{|k|^{2\tau}}I_{(2d)^{\imath+\jmath}}\right\}.
\end{align*}
With Taylor series,
$$A_{\imath\jmath}(\xi)=\mathcal{L}(\xi_0)\Theta(\xi),$$
where
\begin{align*}
\mathcal{L}(\xi_0)&=\left(A_{\imath\jmath}(\xi_0),\partial_\xi A_{\imath\jmath}(\xi_0),\cdots,\partial_\xi^i A_{\imath\jmath}(\xi_0),\int_{0}^1(1-t)^{|i+1|}\partial_\xi^{i+1}A_{\imath\jmath}(\xi_0+t\hat\xi)dt\right),\\
\Theta(\xi)&=\left(I,\hat\xi I,\cdots,\hat\xi^iI,\hat\xi^{i+1}I\right)^\top,\\
\hat\xi&=\xi-\xi_0.
\end{align*}
Noting that $\tilde S_{\imath\jmath}$ comes from the perturbation and {(A1)}, we have
\begin{align}\label{maC}
rank\left\{\mathcal{C}_j: 1\leq j\leq (2d)^{\imath+\jmath}\right\}=(2d)^{\imath+\jmath},\end{align}
where $A_{\imath\jmath}^j$ is the $j$-th column of $A_{\imath\jmath}$, $\mathcal{C}_j$ is a column of $\left\{\partial_\xi^i A_{\imath\jmath}^j:0\leq|i|\leq M\right\}$ and $A_{\imath\jmath}$ is defined as in (\ref{A}).
By (\ref{maC}) and the continuity of determinant, there is an orthogonal matrix $Q_{\xi_0}$ such that
$$\mathcal{L}(\xi)Q_{\xi_0}=\left(E(\xi),F(\xi)\right),~~~for~ \xi\in\bar{G}_{\xi_0},$$
where $E(\xi)$ is a $(2d)^{\imath+\jmath}\times (2d)^{\imath+\jmath}$ nonsingular matrix on $\bar{G}_{\xi_0}$, and $\bar{G}_{\xi_0}$ is the closure of a neighborhood $G_{\xi_0}$ of $\xi_0$.
Moreover,
\begin{align*}
A_{\imath\jmath}&=\mathcal{L}(\xi)Q_{\xi_0}Q_{\xi_0}^{-1}\Theta=(E(\xi),F(\xi))\tilde\Theta,\\
\overline{A_{\imath\jmath}^\top}&=\overline{\tilde{\Theta}^\top}\overline{(E(\xi),F(\xi))^\top},\\
\tilde\Theta&=Q_{\xi_0}^{-1}\Theta=\left(\begin{array}{c}
\tilde\Theta_1\\
\tilde\Theta_2
\end{array}\right),
\end{align*}
where $\tilde\Theta_1$ is a $(2d)^{\imath+\jmath}\times (2d)^{\imath+\jmath}$-order matrix.
Obviously,
$$rank\left(\begin{array}{cc}
\overline{E^\top}E&\overline{E^\top}F\\
\overline{F^\top}E&\overline{F^\top}F\end{array}\right)= (2d)^{\imath+\jmath}.$$
Then there is an unitary matrix $U_\xi$ such that
\begin{align}
U_\xi^{-1}\left(\begin{array}{cc}
\overline{E^\top}E&\overline{E^\top}F\\
\overline{F^\top}E&\overline{F^\top}F\end{array}\right)U_\xi=\left(\begin{array}{cc}
\diag(\lambda_1,\cdots,\lambda_{(2d)^{\imath+\jmath}})&0\\
0&0\end{array}\right),
\end{align}
where $\lambda_i$, $1\leq i\leq (2d)^{\imath+\jmath}$, are nonvanishing eigenvalue of $\left(\begin{array}{cc}
\overline{E^\top}E&\overline{E^\top}F\\
\overline{F^\top}E&\overline{F^\top}F\end{array}\right)$. By Poincar\'{e} Separation Theorem (for example, p. 96, \cite{wang}),
\begin{align*}
\overline{A_{\imath\jmath}^\top}A_{\imath\jmath}&\geq\overline{\tilde{\Theta}^\top}\left(\begin{array}{cc}
\overline{E^\top}E&\overline{E^\top}F\\
\overline{F^\top}E&\overline{F^\top}F\end{array}\right)\tilde{\Theta}\\
&\geq\overline{\tilde{\Theta}^\top}\left(\begin{array}{cc}
\diag(\lambda_1,\cdots,\lambda_{(2d)^{\imath+\jmath}})&0\\
0&0\end{array}\right)\tilde{\Theta}\\
&\geq\left(\begin{array}{cc}\overline{\tilde{\Theta}_1^\top}&\overline{\tilde{\Theta}_2^\top}\end{array}\right)\left(\begin{array}{cc}
\diag(\lambda_1,\cdots,\lambda_{(2d)^{\imath+\jmath}})&0\\
0&0\end{array}\right)\left(\begin{array}{c}
\tilde\Theta_1\\
\tilde\Theta_2
\end{array}\right)\\
&\geq\min_{i}\min_{\bar{\mathcal{O}}_{\xi_0}}\lambda_i\overline{\tilde{\Theta}^\top_1}\tilde{\Theta}_1\\
&\geq\min_{i}\min_{\bar{\mathcal{O}}_{\xi_0}}\lambda_i\left(\min_{j}|\hat{\xi}_j|\right)^{2M+2}I_{(2d)^{\imath+\jmath}},
\end{align*}
where "$\geq$" is L\"{o}wer order (for specific definition, see \cite{wang}) and the last equality has been explicitly explained on page $12$ of \cite{qian3}. Hence, with finite cover theorem, we have
$$|\mathcal{S}_2|=\left|\left\{\xi\in G,0\leq\overline{A_{\imath\jmath}^\top}A_{\imath\jmath}\leq\frac{\gamma^2}{|k|^{2\tau}}I_{(2d)^{\imath+\jmath}}\right\}\right|\leq c\frac{\gamma^{\frac{1}{M+1}}}{|k|^{\frac{\tau}{M+1}}}.$$
Similarly,
$$|\mathcal{S}_1|=\left|\left\{\xi\in G,|\langle k,\omega(\xi)\rangle|\leq\frac{\gamma}{|k|^\tau}\right\}\right|\leq c\frac{\gamma^{\frac{1}{M+1}}}{|k|^{\frac{\tau}{M+1}}}.$$
Then $|\mathcal{R}_{\nu+1}|<\frac{\gamma_0^{\frac{1}{M+1}}}{|k|^{\frac{\tau}{M+1}}}$. Consequently,
\begin{align*}
|G_0\backslash G_*|=\left|\bigcup_{\nu=0}^\infty\bigcup_{K_{\nu}<|k|\leq K_{\nu+1}}\mathcal{R}_k^{\nu+1}\right|\leq\sum_{\nu=0}^\infty\sum_{K_\nu<|k|\leq K_{\nu+1}}\frac{\gamma_0^{\frac{1}{M+1}}}{|k|^{\frac{\tau}{M+1}}}\rightarrow0, as ~\varepsilon\rightarrow0.
\end{align*}

This completes the proof.
\end{proof}}

\section{Appendix A. }\label{appa}
\textbf{Proof of Proposition  \ref{pro1}.}
\begin{proof}
 Notice that
\begin{align*}
\nabla g(u,v)=\left(u|u|^{2l_0-2},v|v|^{2k_0-2}\right),~~~\nabla g(0,0)=0.
\end{align*}
For $0<\delta<1$, $B_\delta(0)$ denotes the open ball centered at the origin with radius $\delta$. We have that $\nabla g(u,v)-\nabla g(0,0)$ is odd and unequal to zero on $\partial B_\delta(0)$, i.e.,
\begin{align*}
\nabla g(-u,-v)-\nabla g(0,0)=\left(-u|u|^{2l_0-2},-v|v|^{2k_0-2}\right)=-(\nabla g(u,v)-\nabla g(0,0)),
\end{align*}
and
\begin{align*}
\nabla g(u,v)-\nabla g(0,0)\neq0,~~\forall (u,v)\in\partial B_\delta(0).
\end{align*}
It follows from Borsuk's theorem in \cite{mot} that
\begin{align*}
\deg\left(\nabla g(u,v)-\nabla g(0,0),B_\delta(0),0\right)\neq0.
\end{align*}
Let $z=(u,v)$. Obviously, there exist $\sigma=\frac{\min_{z\in B_\delta(0)}\{{(z+{\delta_*})\vert z+{\delta_*}\vert^{2\ell}-z\vert z\vert^{2\ell}}\}}{2{\delta_*}^{2\ell+1}}$ and $L=2\max\{l_0,k_0\}-1$  such that
\begin{align*}
\left\vert\nabla g(z)-\nabla g(z_*)\right\vert\geq\sigma\left\vert z-z_*\right\vert^L,~~z_*\in B_\delta(0), z\in B_\delta(0)\setminus B_{\delta_*}(z_*),
\end{align*}
where~$\delta_*>0$, $B_{\delta_*}(z_*)\subset B_\delta(0)$. So, by Theorem \ref{th1}, the results in Proposition \ref{pro1} hold.

\end{proof}
\section{Appendix B. }\label{appb}
\textbf{Proof of Proposition \ref{pro2}}
\begin{proof}
Note that the motion equation of (\ref{HH}) is
\begin{equation*}
\left\{
\begin{array}{ll}
\dot{x}=\omega,\\
\dot{y}=0,\\
\dot{u}=v^2,\\
\dot{v}=-u^2-\varepsilon^2.
\end{array}
\right.
\end{equation*}
Since for any $\varepsilon\neq0$, the equation $u^2+\varepsilon^2=0$ has no real solution, we see that the Hamiltonian system does not admit any lower-dimensional torus. Let $g(u,v)=\frac{1}{3}u^3+\frac{1}{3}v^3$. By simple calculation, we have
$$\deg(\nabla g(u,v),B_\delta(0),0)=0,$$
which implies that {(A0)} fails. Hence our assumption {(A0)} is sharp to Melnikov's persistence.

\end{proof}

\section*{Acknowledgments}
The first author (Jiayin Du)  is
supported by the Fundamental Research Funds for the Central Universities (Grant number 2412024QD003).
The second author (Shuguan Ji) is supported by National Natural Science Foundation of China (Grant numbers 12225103, 12071065, 11871140),
and the National Key Research and Development Program of China (Grant numbers 2020YFA0713602, 2020YFC1808301). The third author (Yong Li) is supported by National Basic Research Program of China (Grant number 2013CB834100), National Natural Science Foundation of China (Grant numbers 11571065, 11171132, 12071175, 12471183),
and Natural Science Foundation of Jilin Province (Grant number 20200201253JC).





\end{document}